\documentclass[12pt, letterpaper]{amsart}
\usepackage{amsmath}
\usepackage{tikz}
\usetikzlibrary{matrix,arrows,decorations.pathmorphing} 
\usepackage{graphicx}
\usepackage{amssymb}
\usepackage{mathrsfs}
\usepackage{latexsym}
\usepackage[
hypertexnames=false, colorlinks, citecolor=red, linkcolor=red]{hyperref}

	\normalbaselines						  %

\newcommand{\R}{{\mathbb  R}}

\newcommand{\D}{{\mathbb  D}}

\newcommand{\T}{\mathbb{T}}

\newcommand{\Z}{{\mathbb  Z}}
\newcommand{\N}{{\mathbb  N}}
\newcommand{\C}{{\mathbb  C}}

\newcommand{\E}{\mathbb E}

\newcommand{\OZ}{{\mathbf{0}}}
\newcommand{\dd}{{\mathrm{d}}}

\newcommand{\OID}{{\mathbf{I}}}
\newcommand{\bI}{\mathbf{I}}

\newcommand{\fdot}{\,\cdot\,}

\newcommand{\uG}{^\Gamma}
\newcommand{\ciG}{\ci\Gamma}

\newcommand{\cP}{\mathcal{P}}
\newcommand{\cC}{\mathcal{C}}
\newcommand{\cT}{\mathcal{T}}
\newcommand{\cS}{\mathcal{S}}

\newcommand{\bO}{\mathbf{0}}

\newcommand{\bB}{\mathbf{B}}
\newcommand{\bM}{\mathbf{M}}
\newcommand{\bN}{\mathbf{N}}
\newcommand{\be}{\mathbf{e}}

\newcommand{\bc}{\mathbf{c}}
\newcommand{\bb}{\mathbf{b}}

\newcommand{\wt}{\widetilde}

\newcommand{\cK}{{\mathcal K}}

\newcommand{\cH}{\mathcal{H}}
\newcommand{\cD}{\mathcal{D}}

\newcommand{\cU}{\mathcal{U}}
\newcommand{\cE}{\mathcal{E}}

\newcommand{\cW}{\mathcal{W}}

\newcommand{\cF}{\mathcal{F}}
\newcommand{\1}{\mathbf{1}}

\newcommand{\f}{\varphi}

\newcommand{\e}{\varepsilon}

\DeclareMathOperator{\spn}{span}
\DeclareMathOperator{\cspn}{\overline{span}}

\DeclareMathOperator{\tr}{tr}

\DeclareMathOperator{\Ran}{Ran}

\DeclareMathOperator{\im}{Im}
\DeclareMathOperator{\re}{Re}

\DeclareMathOperator{\rank}{rank}
\DeclareMathOperator*{\s-lim}{\text{\rm s-}lim}
\DeclareMathOperator*{\w-lim}{\text{\rm w-}lim}
\DeclareMathOperator*{\wot-lim}{\text{\rm w.o.t.-}lim}

\newcommand{\ci}[1]{_{_{\scriptstyle #1}}}

\newcommand{\ti}[1]{_{\scriptstyle \text{\rm #1}}}

\newcommand{\ut}[1]{^{\scriptstyle \text{\rm #1}}}

\catcode`@=11
\chardef\mathlig@atcode\count255

\def\actively#1#2{\begingroup\uccode`\~=`#2\relax\uppercase{\endgroup#1~}}
\def\mathlig@gobble{\afterassignment\mathlig@next@cmd\let\mathlig@next= }
\def\mathlig@delim{\mathlig@delim}
\def\mathlig@defcs#1{\expandafter\def\csname#1\endcsname}
\def\mathlig@let@cs#1#2{\expandafter\let\expandafter#1\csname#2\endcsname}
\def\mathlig@appendcs#1#2{\expandafter\edef\csname#1\endcsname{\csname#1\endcsname#2}}

\def\mathlig#1#2{\mathlig@checklig#1\mathlig@end\mathlig@defcs{mathlig@back@#1}{#2}\ignorespaces}

\def\mathlig@checklig#1#2\mathlig@end{%
 \expandafter\ifx\csname mathlig@forw@#1\endcsname\relax
 \expandafter\mathchardef\csname mathlig@back@#1\endcsname=\mathcode`#1%
 \mathcode`#1"8000\actively\def#1{\csname mathlig@look@#1\endcsname}%
 \mathlig@dolig#1\mathlig@delim
\fi
\mathlig@checksuffix#1#2\mathlig@end
}

\def\mathlig@checksuffix#1#2\mathlig@end{%
\ifx\mathlig@delim#2\mathlig@delim\relax\else\mathlig@checksuffix@{#1}#2\mathlig@end\fi
}
\def\mathlig@checksuffix@#1#2#3\mathlig@end{%
\expandafter\ifx\csname mathlig@forw@#1#2\endcsname\relax\mathlig@dosuffix{#1}{#2}\fi
\mathlig@checksuffix{#1#2}#3\mathlig@end
}

\def\mathlig@dosuffix#1#2{%
\mathlig@appendcs{mathlig@toks@#1}{#2}%
\mathlig@dolig{#1}{#2}\mathlig@delim
}

\def\mathlig@dolig#1#2\mathlig@delim{%
 \mathlig@defcs{mathlig@look@#1#2}{%
 \mathlig@let@cs\mathlig@next{mathlig@forw@#1#2}\futurelet\mathlig@next@tok\mathlig@next}%
 \mathlig@defcs{mathlig@forw@#1#2}{%
  \mathlig@let@cs\mathlig@next{mathlig@back@#1#2}%
  \mathlig@let@cs\checker{mathlig@chck@#1#2}%
  \mathlig@let@cs\mathligtoks{mathlig@toks@#1#2}%
  \expandafter\ifx\expandafter\mathlig@delim\mathligtoks\mathlig@delim\relax\else
  \expandafter\checker\mathligtoks\mathlig@delim\fi
  \mathlig@next
 }%
 \mathlig@defcs{mathlig@toks@#1#2}{}%
 \mathlig@defcs{mathlig@chck@#1#2}##1##2\mathlig@delim{%
  \ifx\mathlig@next@tok##1%
   \mathlig@let@cs\mathlig@next@cmd{mathlig@look@#1#2##1}\let\mathlig@next\mathlig@gobble
  \fi
  \ifx\mathlig@delim##2\mathlig@delim\relax\else
   \csname mathlig@chck@#1#2\endcsname##2\mathlig@delim
  \fi
 }%
 \ifx\mathlig@delim#2\mathlig@delim\else
  \mathlig@defcs{mathlig@back@#1#2}{\csname mathlig@back@#1\endcsname #2}%
 \fi
}%

\catcode`@\mathlig@atcode

\mathchardef\ordinarycolon\mathcode`\:
\def\vcentcolon{\mathrel{\mathop\ordinarycolon}}
\mathlig{:=}{\vcentcolon=}
\mathlig{::=}{\vcentcolon\vcentcolon=}

\renewcommand{\labelenumi}{(\roman{enumi})}

\newcounter{vremennyj}

\newcommand\cond[1]{\setcounter{vremennyj}{\theenumi}\setcounter{enumi}{#1}\labelenumi\setcounter{enumi}{\thevremennyj}}

\numberwithin{equation}{section}

\theoremstyle{plain}
\newtheorem{theo}{Theorem}[section]
\newtheorem{cor}[theo]{Corollary}
\newtheorem{lem}[theo]{Lemma}
\newtheorem{prop}[theo]{Proposition}

\theoremstyle{definition}
\newtheorem{defn}[theo]{Definition}

\theoremstyle{remark}
\newtheorem*{ex*}{Example}
\newtheorem*{exs*}{Examples}
\newtheorem{rem}[theo]{Remark}
\newtheorem*{rem*}{Remark}
\newtheorem*{rems*}{Remarks}

\title[Matrix Measures and Finite Rank Perturbatios]
{Matrix Measures and Finite Rank Perturbations of Self-adjoint Operators}

\author{Constanze~Liaw}
\address{C.~Liaw: Department of Mathematical Sciences, Ewing 311, University of Delaware, Newark, DE 19716, USA; and 
CASPER, Baylor University, One Bear Place \#97328,      
 Waco, TX  76798, USA.}
\email{liaw@udel.edu}
\author{Sergei~Treil}
\address{S.~Treil: Department of Mathematics, Brown University   
151 Thayer
Str./Box 1917,      
 Providence, RI  02912, USA}
 
\email{treil@math.brown.edu}

 \thanks{
Work of C.~Liaw was supported by the National Science Foundation under the grant  DMS-1802682.} 
\thanks{Work of S.~Treil is supported by the National Science Foundation under the grants  DMS-1600139.
}

\begin{document}

\begin{abstract}
	
Matrix-valued measures provide a natural language for the theory  of finite rank perturbations. 
In this paper we use this language to prove some new perturbation theoretic results.

Our main result is a generalization of the Aronszajn--Donoghue theorem about the mutual singularity of the singular parts of the spectrum for rank one perturbations to the case of finite rank perturbations. Simple direct sum type examples  indicate that  an exact generalization is not possible.  However, in this paper we introduce the notion of \emph{vector mutual singularity} for the matrix-valued measures and show that if  we use this notion, the mutual singularity still holds for the finite rank perturbations.

As for the scalar spectral measures and the classical mutual singularity, we show that the singular parts are mutually singular for almost all perturbations. One of the ways to prove that is to use a generalization of the  Aleksandrov's spectral averaging to the matrix-valued measures, which  is also one of the main results of this paper. 

Finally, the spectral representation of the perturbed operator is obtained. The matrix Muckenhoupt $A_2$ condition appears naturally there, and it plays an important role in establishing the vector mutual singularity of the spectral measures.   
\end{abstract}

\maketitle
\setcounter{tocdepth}{1}
\tableofcontents

\setcounter{tocdepth}{2}
\section{Introduction}
The theory of rank one perturbations can be traced back to a seminal paper in 1910 by Weyl \cite{weyl}, where they were introduced as a tool to determine the spectrum of Sturm--Liouville operators when objected to changing boundary conditions.

Most of the spectral behavior under rank one perturbations is very well understood and can be easily obtained by the analysis of the Cauchy transforms of the corresponding spectral measures.
One of the consequences of this analysis is a classical Aronszajn--Donoghue theorem\footnote{This result was proved by Aronszajn for Sturm--Liouville operators with varying boundary conditions \cite{Aronszajn} and by Donoghue in the abstract setting of rank one perturbations \cite{Donoghue}.}, which states that the singular parts of the spectral measures from the family of the perturbed operators are mutually singular.  

%

The situation in the case of finite rank perturbations is less understood.
While the Kato--Rosenblum theorem holds for trace class perturbations, the Aronszajn--Donoghue theory is not developed. And simple direct sum type examples suggest that a result like the mutual singularity of singular parts should not be possible in the finite rank case. 

In this paper we consider matrix-valued spectral measures, that seem to be the natural objects in the case of higher rank perturbations. The language of matrix- and operator-valued spectral measures was developed earlier in the theory, see for example \cite{deBranges, Kuroda1967, Kuroda1963}, but became less popular later on. For the perturbations by rank $d$ operators the corresponding spectral measures take values in the space of $d\times d$ positive semidefinite matrices; very often the density is degenerate a.e.

Using such matrix-valued spectral measures we show our main theorem --- that mutual singularity of singular parts holds for the finite rank perturbations, if by mutual singularity one understands \emph{vector} mutual singularity of the matrix-valued measures,  see Definition \ref{d:vector sing} and Theorem \ref{t-MMS} below. 

The proof is rather interesting: we first establish a formula for the spectral representation of the perturbed operator, see Theorem \ref{t:repr-01} below. This representation formula implies the two weight estimates for the Cauchy transform, which in turn implies the matrix Muckenhoupt $A_2$ condition for the pair of the spectral measures, see Theorem \ref{t: M-M_Gamma A_2} below. 
The vector mutual singularity of the singular parts of the matrix-valued measures is then a simple corollary of this $A_2$ condition.  

Another interesting result in rank one perturbation theory is the Aleksandrov disintegration theorem, stating that averaging the spectral measures of the family of rank one perturbations gives us  the Lebesgue measure. In Theorems \ref{Disintegration 01} and \ref{t-AvHaar}, we prove a version of this result for the case of finite rank perturbations; some interesting new phenomena appear in the statement and in the proof of this result. 

The matrix version of the Aleksandrov disintegration theorem allows us to get a type of mutual singularity result for singular parts for the scalar spectral measures. Namely, we are able to show that the singular parts of the \emph{scalar} spectral measures are mutually singular with the singular parts of the unperturbed operators for almost all perturbations, see Corollary \ref{c-MutuallySingular} below. 

\subsection{Plan of the paper}
Section \ref{s-FINITE} is devoted to a basic set up of finite rank self-adjoint perturbations and their matrix-valued spectral measures. We include known results on these measures and cyclic subspaces.

In Section \ref{s-Borel} we present well-known basic facts in
 perturbation theory: an Aronszajn--Krein type formula relating the Cauchy transforms of the spectral measures $\bM$ and $\bM\uG$, and the relationship between non-tangential (upper half-plane) boundary values of the Cauchy transform and its matrix-valued measure.

The results in Sections \ref{s-FINITE} and \ref{s-Borel} are well-known, see e.g.~\cite{Kuroda1967, Yafaev1992}: we present the proofs only for the reader's convenience,  to make the paper self-contained.

Certain generalizations of the Aleksandrov spectral averaging to matrix-valued spectral measures are proved in Section \ref{s-SAveraging}. The averaging formulas are then used to assert restrictions on the singular spectrum.

Section \ref{s-Repr} features a spectral representation formula in the spirit of the authors' paper \cite{LT09}.
This representation is then used in Section \ref{s-MMS} to show that the singular parts of the matrix-valued measures $\bM$ and $\Gamma\bM\uG\Gamma$, where $\bM$ and $\bM^\Gamma$ are the matrix-valued spectral measures of $A$ and the perturbed operator $A\ci\Gamma$,  are what we call \emph{vector mutually singular}. This is one of the main results of the paper, and it should be thought of as a generalization of the Aronszajn--Donoghue theorem to higher rank perturbations. The proof involves the matrix Muckenhoupt $A_2$-condition.

As it is well known to experts, the technique of matrix-valued measures can be used to prove many standard results of the perturbation theory. In Appendix \ref{s:KR} we present a proof of the Kato--Rosenblum theorem, based on the technique developed in this paper. While main ideas of the proof are well-known to experts, the proof could be of interest to non-specialists. 

%

\section{Finite rank perturbations}\label{s-FINITE}
Let $A$ be a self-adjoint operator on a separable Hilbert space $\cH$. Motivated by the theory of self-adjoint extensions of a symmetric operator with deficiency indices $d$, we fix a $d$ dimensional subspace $\cK$ of $\cH$ and consider all self-adjoint perturbations of $A+T$ that satisfy $\Ran T\subset \cK$.

Such operators $A+T$ can be conveniently parametrized using $d\times d$ matrices. To realize this parametrization, we fix a left invertible operator ${\bB}:\C^d\to \cH$, $\Ran\bB=\cK$. Define 
\begin{align*}
b_k:= \bB \be_k, \qquad k=1, 2, \ldots , d, 
\end{align*}
where $\be_1, \be_2, \ldots \be_d$ is the standard orthonormal basis in $\C^d$. 

This family of rank $d$ perturbations is now formally associated with
\begin{align}\label{e-AGamma}
A\ci\Gamma = A + {\bf B}\Gamma{\bf B}^*
\qquad \text{on }D(A)
\end{align}
where the $d\times d$ matrix $\Gamma$ is self-adjoint; the family of perturbations can be rigorously defined through resolvents or quadratic forms.

For simplicity the reader can assume that to operator $\bB$ is  bounded. 
However, everything works for the 
 \emph{(singular) form bounded perturbations}; that means that while $\bB$ can be unbounded,  for each $k$ we have $\|(1+|A|)^{-1/2}b_k\|\ci\cH<\infty$ where $|A| = (A^*A)^{1/2}$ is the modulus of $A$. In other words, the operator $(1+|A|)^{-1/2}\bB$ should be bounded.  Many applications to differential equations fall into this category. While more singular perturbations are possible (see \cite{kurasovbook}), they are not uniquely defined and instead require another parameter choice. 
In Remark \ref{r-FormBdd} below we mention a characterization of form boundedness in terms of the spectral measure.

Below, we will not assume that $\Gamma$ is invertible. In situations when we do require invertibility, we will explicitly mention it.

Focussing on the non-trivial part of the perturbation problem we assume that $\cK$ is a cyclic subspace for $A$, i.e.~$\cH = \cspn\{(A-z)^{-1}b:z\in \C\setminus\R, b\in \cK\}$. This assumption does not essentially restrict generality. Indeed, without this assumption, the restrictions of $A\ciG$ and $A$ to the orthogonal complement, $$\widehat{\cH} =  ( \cspn\{(A-z\OID)^{-1}b:z\in \C\setminus\R, b\in \cK\})^\perp,$$ (in the possibly larger $\cH$) are equal. That is, $A\ci\Gamma|\ci{\widehat{\cH}} = A|\ci{\widehat{\cH}}$.

Cyclic subspaces for $A$ are characterized in Lemma \ref{l:cyclicity}. In Lemma \ref{l-cyc} we prove a well-known fact stating that a cyclic subspace for $A$ is also cyclic for all perturbed operators $A\ci\Gamma$.

\subsection{Spectral representation in the von Neumann direct integral}
\label{ss:von Neumann}
By the spectral theorem a self-adjoint operator is unitarily equivalent to the multiplication operator $M_t$ by the independent variable $t$, $M_t f (t) = tf(t)$ in the von Neumann direct integral
\begin{align}
\label{direct int 01}
\cH= \int_\R \oplus H(t)\dd\mu(t); 
\end{align}
here $\mu$ is a scalar spectral measure of the operator.

Let us recall the construction of the von Neumann direct integral. We start with a separable Hilbert space $H$ with an orthonormal basis $(e_k)_{k\ge1}$, and a measurable function $N:\R\to \Z_+ \cup \{+\infty\}$. This dimension function $N$ indicates the multiplicity of the spectrum. (For example, when considering rank one perturbations, we have $N\equiv1$ a.e.~with respect to $\mu$.)

Define  
\begin{align*}
H(t)=\cspn\{e_k:1\le k \le N(t)\} . 
\end{align*}
Then the von Neumann direct integral \eqref{direct int 01} is defined as 
\begin{align*}
\cH :=\{f\in L^2(\mu;H): f(t)\in H(t ) \ \mu \text{-a.e.} \}.
\end{align*}

For a measure $\mu$ let the \emph{spectral class} be the set of all measures mutually absolutely continuous with respect to $\mu$.  
We will need the following well-known fact, cf.~\cite[Ch.~7, Theorem 5.2]{BirmanSol-book_1987}.

\begin{theo}
	\label{t:SpectrInv}
	The spectral class of the scalar spectral measure and the dimension function $N$ completely define a self-adjoint operator up to unitary equivalence. 
	
	Namely, two self-adjoint operators (represented in the von Neumann direct integrals with measures $\mu$ and $\mu_1$, and the dimension functions $N$ and $N_1$ respectively) are unitarily equivalent if and only if the measures $\mu$ and $\mu_1$ are mutually absolutely continuous and $N(t) = N_1(t)$ $\mu$-a.e.
\end{theo}

\subsection{Matrix-valued spectral measures and spectral representations}\label{ss-MatrixMeasures}
In this paper by a matrix-valued measure we will understand a countably additive set function (defined on bounded  Borel subsets of $\R$) with values in the set of $d\times d$ Hermitian positive semidefinite matrices (with complex entries). Here we always assume that the measure is Radon, i.e.~that it is bounded on bounded Borel subsets of $\R$. 

A matrix measure $\bM$ can be represented as a matrix $(\mu_{j,k})_{j,k=1}^d$, where $\mu_{j,k}$ are Radon measures on $\R$; the measures $\mu_{k,k}$ are non-negative, and the measures $\mu_{j,k}$ can be complex-valued. The fact that $\bM$ takes values in the set of positive semidefinite matrices simply means that for any bounded Borel set $E$ the matrix $(\mu_{j,k}(E))_{j,k=1}^d$ is Hermitian positive semidefinite. 

For a matrix-valued measure $\bM$ define the scalar measure $\mu=\tr \bM = \sum_{k=1}^d \mu_{k,k}$. Since $\bM(E)$ is positive semidefinite, we get that $|\mu_{j,k}|\le \frac12 (\mu_{j,j} + \mu_{k,k} )$. Therefore, the measures $\mu_{j,k}$ are absolutely continuous with respect to $\mu$, $|\mu_{j,k}|\le \mu$, so the matrix measure $\bM$ is absolutely continuous with respect to $\mu$, $\dd\bM = W\dd\mu$, where $W$ is a measurable matrix-valued functions with values in the set of positive semidefinite Hermitian matrices. Moreover, if $\mu=\tr \bM$, then  $W\in L^\infty$. 

Given a matrix-valued measure $\bM$, we can define the weighted space $L^2(\bM)=L^2(\R, \bM; \C^d)$ of $\C^d$-valued measurable functions $f$ such that
\begin{align*}
\|f\|\ci{L^2(\bM)}^2 := \int_\R \left( [\dd \bM(t)] f(t), f(t) \right)\ci{\C^d} = \int_\R \left( W(t) f(t), f(t) \right)\ci{\C^d}\dd\mu(t) . 
\end{align*}
The vector-valued  integral $\int [\dd\bM] f$ is naturally defined as
\begin{align*}
\int_\R [\dd\bM]f = \int_\R W(t)f(t)\dd\mu(t).
\end{align*}

\subsubsection{Matrix-valued spectral measures}
Let $\cE$ be the projection-valued spectral measure of $A$. Define a matrix-valued measure $\bM$ (with values in the set of $d\times d$ positive semidefinite matrices) by 
\begin{align}
\label{dM 01}
\bM(E) = \bB^* \cE(E) \bB\qquad \text{for all Borel } E\subset \R. 
\end{align}
Equivalently, this can be rewritten as 
\begin{align}\label{d-M}
{\bf B}^* (A-z\OID)^{-1} {\bf B} = \int_\R \frac{\dd \bM(t)}{t-z}
\end{align}
for all $z\in\C\setminus \R$.
Equation \eqref{d-M} can be used when considering general (possibly unbounded) operators $A$ and a set of vectors that generates a cyclic subspace. 

\begin{rem}\label{r-FormBdd}
	It is easy to see that the
 perturbation $\bB\Gamma\bB^*$ with invertible $\Gamma$ is bounded if and only if  the spectral measure $\bM$ is finite ($\bM(\R)\le c\bI$), and it is form bounded  if and only it 
$
 \int_\R \frac{\dd \bM(t)}{|t|+1} <c \OID
$ for some $c<\infty$.
\end{rem}

If $\Ran \bB$ is cyclic, meaning that $ \cspn \{(A- z \OID)^{-1} b_k: k=1, 2, \ldots, d, \ z \in \C\setminus \R  \} = \cH$, the operator $A$ is unitarily equivalent to the multiplication $M_t$ by the independent variable $t$ in the weighted space $L^2(\bM)=L^2(\R,\bM;\C^d)$. The intertwining unitary map $\cU:  L^2(\bM) \to \cH $ is given by 
\begin{align}
\label{UnitEquiv 01}
 (t-z)^{-1}\be_k \mapsto (A-z\OID)^{-1}b_k  = (A-z\OID)^{-1}\bB\be_k, 
\end{align}  
where, recall,  $(\be_k)_{k=1}^d$ is the standard orthonormal basis in $\C^d$. 
It is easy to see that $\cU$ is an isometry, and cyclicity of $\Ran \bB$ implies that $\cU$ is unitary. 

If $A$ is given in its standard spectral representation, i.e.~it is represented as a multiplication $M_t$ by the independent variable 
$t$ in the von Neumann  direct integral \eqref{direct int 01}.

In this case the operator $\bB$ acts through multiplication by the matrix-valued function $B$, $B(t) : \C^d\to H(t)$, 
\begin{align*}
(\bB\be) (t) = B(t)\be(t), \qquad \be\in\C^d; 
\end{align*}
the vector $b_k(t)\in H(t)$ is the $k$th column of the matrix $B(t)$. 

The above  unitary operator \eqref{UnitEquiv 01} 
can then be rewritten  as 
\begin{align*}
[\cU h\be] (t)  = h(t) B(t)\be, \qquad \be\in \C^d, \ h \text{ is a scalar-valued function}. 
\end{align*}
Using the density of the linear combinations of the functions of form $h\be$ in $L^2(\bM)$ we  obtain the 
representation 
\begin{align}
\label{UnitaryEquiv 03}
[ \cU f](t) = B(t)f(t), \qquad f\in L^2(\bM).
\end{align}
Since $\cU$ is a unitary operator (and thus surjective), the above representation \eqref{UnitaryEquiv 03} implies the following simple lemma. 
\begin{lem}
	\label{l:cyclicity}
	$\Ran\bB$ is cyclic for $A$ if and only if 
	\begin{align*}
	\Ran B(t) = \spn\{b_k(t): 1\le k\le d\} = H(t) \qquad \mu\text{-a.e.}
	\end{align*}
\end{lem}

Since $\cU$ is unitary, we get from \eqref{UnitaryEquiv 03}  (assuming that $\dd\bM = W\dd\mu$, and the same measure $\mu$ is used in the von Neumann direct integral \eqref{direct int 01}) that
\begin{align}
\label{W=B^*B}
W(t) = B^*(t) B(t) \qquad \mu\text{-a.e.};
\end{align}
if in \eqref{direct int 01} a different measure $\mu_1$ is used, then the right hand side of \eqref{W=B^*B} should be multiplied by the density $\dd\mu_1/\dd\mu$.

By Lemma \ref{l:cyclicity} $\rank B(t)= \dim H(t)$ $\mu$-a.e.; combining this with Theorem \ref{t:SpectrInv} we obtain the following simple statement. 

\begin{prop}
\label{p:UnitaryInv 02}
Let $\bM=W\mu$ and $\bN=V\nu$ be the matrix-valued spectral measures and let $A$ and $B$ be the multiplication  operators by the independent variable $t$ in $L^2(\bM)$ and $L^2(\bN)$ respectively. Then $A$ and $B$ are unitarily equivalent if and only if the scalar measures $\mu$ and $\nu$ are mutually absolutely continuous and 
\begin{align*}
\rank W(t) =\rank V(t)\qquad \mu\text{-a.e.}
\end{align*}
\end{prop}

\begin{rem*}
Note, that in the above proposition we do not require that the matrices $\bM$ and $\bN$  are of the same size. 
\end{rem*}

For the matrix spectral measure $\bM$ its density $W$ does not need to be full rank; if, for example, $A$ has a cyclic vector, then $\rank W(t)=1$ $\mu$-a.e. More generally, if we have a spectral representation in the von Neumann  direct integral 
\begin{align*}
\int_\R \oplus H(t)\dd\mu(t), 
\end{align*}
then $\rank W(t) = \dim H(t)$ $\mu$-a.e.

\subsection{Spectral representation with matrix spectral measures for \texorpdfstring{$A\ci\Gamma$}{A<sub>Gamma}}\label{ss-MGamma}

For the perturbed operator $A\ci\Gamma$ given by \eqref{e-AGamma} we can similarly define the matrix-valued spectral measure $\bM^\Gamma$ by 
\begin{align}
\label{M^Gamma}
\bB^* (A\ci\Gamma - z \OID)^{-1} \bB = \int_\R \frac{\dd \bM^\Gamma(t)}{t-z} =: F\ci\Gamma(z)  \qquad \forall z\in \C\setminus\R, 
\end{align}
or equivalently $\bM^\Gamma(E) = \bB^* \cE^\Gamma(E) \bB$, where $\cE^\Gamma$ is the projection-valued spectral measure of $A\ci\Gamma$. 

Since $\Ran \bB$ is cyclic for $A\ci\Gamma$, see Lemma \ref{l-cyc} below, the operator $A\ci\Gamma$ is unitarily equivalent to the multiplication $M_t$ by the independent variable $t$ in the weighted space $L^2(\bM^\Gamma)$; the intertwining unitary operator $\cU\ci\Gamma: L^2(\bM)\to \cH$ is given by \eqref{UnitEquiv 01} with $A$ replaced by $A\ci\Gamma$. 

Similarly to the case of unperturbed operator $A$, define the scalar spectral measure $\mu^\Gamma = \tr \bM^\Gamma$, as well as the matrix weight $W^\Gamma$, $\dd\bM^\Gamma = W^\Gamma\dd\mu^\Gamma$.

\begin{lem}\label{l-cyc}
	Let $\Ran {\bf B}$ be  cyclic  for $A$. And let $A\ciG$ be the family of rank $d$ self-adjoint perturbations, i.e.~$A\ciG = A + {\bf B} \Gamma {\bf B}^*$ for hermitian $d\times d$ matrix $\Gamma$. Then $\Ran B$ is  cyclic for all $A\ciG$.
\end{lem}

Versions of this result go back to early work on scattering theory, see e.g.~\cite[Sec.~2]{kato1957}.

\begin{proof}[Proof of Lemma \ref{l-cyc}]
	Let us use the standard notation for the resolvent
	\begin{align*}
	R\ciG = R\ciG(z) = (A\ci\Gamma-z \OID)^{-1}, 
	\qquad\text{and}\qquad
	R(z) = R\ci\OZ (z)= (A-z \OID)^{-1} .
	\end{align*}
	Take $f\in \cH$. The cyclicity of $\Ran\bB$ for $A$ means that any such $f$ can be approximated by linear combinations of $R(z) b_k$, $z\in\C\setminus\R$, $b_k=\bB\be_k$. Therefore, in order to show the cyclicity of $\Ran\bB$ for $A\ci\Gamma$, it suffices to show that for each $z\in \C\setminus\R$, $1\le k\le d$ the vector $R(z) b_k$ belongs to $R\ci\Gamma(z) \Ran\bB$. To see this, we re-write the resolvent identity
	\begin{align}\label{e-ResId}
	R\ciG = R - R\ciG {\bf B}\Gamma {\bf B}^* R
	\end{align}
	and apply it to $b_k$:
	\[
	R (z)b_k
	= 
	R\ci\Gamma (z) [\OID  + {\bB}\Gamma {\bB}^* R(z) ] b_k.
	\]
	It remains to point out that $[\OID + {\bf B}\Gamma {\bf B}^* R]b_k\in \Ran\bB$.  
\end{proof}

\begin{rem*}
The standard proof (by straightforward algebra) of the resolvent identity \eqref{e-ResId} works for bounded operators $A$. Without going into detail, we point out that the identity extends to form bounded perturbations.
\end{rem*}

\section{Cauchy transform and spectral measures}\label{s-Borel}
Much of the perturbation theory for rank one perturbations relies on relating the Cauchy transform corresponding to $A$ with that corresponding to the perturbed operator. In the case of finite rank self-adjoint perturbations $A\ci\Gamma  = A+{\bf B}\Gamma {\bf B}^*$, we work with matrix-valued Cauchy transforms. Namely, we define the matrix-valued analytic function
\[
F\ci\Gamma(z)
:= \int_\R \frac{\dd \bM^\Gamma(t)}{t-z} =
{\bf B}^* (A\ci\Gamma-z\OID)^{-1} {\bf B}
\qquad\text{for }z\in \C\setminus\R.
\]
For $\Gamma = \OZ$ we abbreviate $F := F\ci\OZ$.

Again, we can obtain an Aronszajn--Krein type relationship between the Cauchy transforms $F\ciG$ and $F$.

The following three lemmata are well-known to experts, see e.g.~\cite{Yafaev1992, Kuroda1967, katokuroda}. We provide complete proofs for the convenience of the reader.

\begin{lem}\label{l-AK}
Let $\Ran \bB$ be cyclic for $A$. Then for all $z\in \C\setminus\R$ and all Hermitian matrices $\Gamma$ the matrices $\bI+F(z)\Gamma$, $\bI + \Gamma F(z)$ are invertible, and 
\begin{align}
\label{F_Gamma}
F\ciG
=
(\OID +  F \Gamma)^{-1} F
=
F(\OID + \Gamma F)^{-1}.
\end{align}
Note that the inverse exists on $\C\setminus \R$.
\end{lem}

\begin{proof}
The resolvent identity says
\begin{align*}
(A\ci\Gamma - z\OID)^{-1}
& =
(A- z\OID)^{-1}
-
(A- z\OID)^{-1}
{\bB}\Gamma{\bB}^*
(A\ci\Gamma - z\OID)^{-1}
\\
& = (A- z\OID)^{-1} - (A\ci\Gamma- z\OID)^{-1}
{\bB}\Gamma{\bB}^*
(A - z\OID)^{-1}.
\end{align*}

Right and left multiplying the first identity by ${\bB^*}$ and ${\bB}$ respectively  and recalling that  $F(z) = {\bB}^*(A-z\OID)^{-1} {\bB}$, $F\ci\Gamma(z) = {\bB}^*(A\ci{\Gamma}-z\OID)^{-1} {\bB}$ we get
\[
F\ci\Gamma = 
F - F\Gamma F\ci\Gamma , 
\]
or, equivalently 
\begin{align}
\label{F v F_G}
(\bI + F \Gamma)F\ci\Gamma = F. 
\end{align}
From here we get by simple algebra that 
\begin{align}
\label{F v F_G-01}
(\bI + F(z) \Gamma)(\bI - F(z)\ci\Gamma\Gamma )\equiv \bI,  \qquad \forall z\in\C\setminus \R, 
\end{align}
which implies that 
 that for all $z\in\C\setminus\R$ the matrices $\bI +\Gamma F(z)$  are invertible for all $z\in \C_+\setminus\R$ (the matrices are square, so one-sided invertibility is equivalent to the invertibility).

Left multiplying  \eqref{F v F_G} by  $(\bI +F(z) \Gamma)^{-1}$ gives the first equality. 

The second formula together with the invertibility of $\bI + \Gamma F(z)$ follows similarly from the second resolvent identity.
\end{proof}

A matrix-valued analytic function $F$ on the upper half-plane is said to be \emph{Herglotz}, if for all $z\in \C_+$ the matrix $F(z)$ is positive semidefinite.

\begin{lem}
The matrix-valued functions $F\ciG$ are Herglotz for all self-adjoint $\Gamma$.
\end{lem}

\begin{proof}
For $\Gamma = \OZ$ we have
\begin{align*}
\im F(z)
&
=
[F(z)-F(z)^*]/2i
=
{\bf B}^* [(A-z\OID)^{-1}-(A-\bar z\OID)^{-1}] {\bf B}/2i
\\&
=
{\bf B}^*(A-z\OID)^{-1} [A-\bar z\OID-(A- z\OID)] (A-\bar z\OID)^{-1}{\bf B}/2i
\\&
=
{\bf B}^*(A-z\OID)^{-1}\im z (A-\bar z\OID)^{-1}{\bf B}.
\end{align*}
That $F$ is Herglotz can now be seen by taking $(\im F(z) \be, \be)\ci{\C^d}$ for $\be\in \C^d$.

For general $\Gamma$ one just need to replace $A$ by $A\ci \Gamma $ in the above formula. 
\end{proof}

We need the following simple lemma, relating $\im F\ci\Gamma$ and $\im F$. 
\begin{lem}
	\label{l:Im F vs Im F Gamma}
For $F$ and $F\ci\Gamma$ defined above
\begin{align}
\label{e-AKIm2}
\im F\ciG(z)
&
=
(\bI+F(z)^*{\Gamma})^{-1}
\im F(z)
(\bI+{\Gamma} F(z))^{-1}  \\ \notag
&=
(\bI+F(z){\Gamma})^{-1}
\im F(z)
(\bI+{\Gamma} F(z)^*)^{-1}.
\end{align}
\end{lem}
\begin{proof}
Using the second identity from Lemma \ref{l-AK} we obtain 
\begin{align*}
\im F\ci\Gamma & = (2i)^{-1} \left( F\ci\Gamma - F\ci\Gamma^*  \right)
= (2i)^{-1} \left( F(\bI + \Gamma F)^{-1}  - (\bI + F^*\Gamma)^{-1}F^* \right) \\
&= (2i)^{-1} (\bI + F^*\Gamma)^{-1}[(\bI + F^*\Gamma)F - F^* (\bI + \Gamma F) ] (\bI + \Gamma F)^{-1}\\
& = (\bI + F^*\Gamma)^{-1} \im F (\bI + \Gamma F)^{-1},  
\end{align*}	
which is the first identity in \eqref{e-AKIm2}. 

The second identity in \eqref{e-AKIm2} is obtained similarly from the first identity in Lemma \ref{l-AK}. 
\end{proof}

\subsection{Retrieving spectral information from Cauchy transforms}
We need the following well-known result connecting boundary behavior of the Poisson extension of a measure to its Radon--Nikodym derivative. 

For a (possibly complex-valued) measure $\tau$ on $\R$ denote by $\tau(z)$ its Harmonic extension to a point $z\in\C\setminus\R$. We assume here that the Poisson extension is well defined, i.e.~that $\int_\R (1+x^2)^{-1}\dd|\tau |(x)  <\infty$. If $\dd\tau =f\dd\mu$, where $f$ is a scalar function, we use the notation $[f\mu](z)$ or $(f\mu)(z)$ to denote the Poisson extension of $f\mu$.

\begin{theo}
\label{t:BoundValues}
\  
\begin{enumerate}
	\item Let measure $\mu\ge 0$ and a measurable function $f$ be such that the Poisson extensions of $\mu$ and $f\mu$ are well defined. Then the non-tangential limit
	\begin{align*}
	\lim_{z\to x\sphericalangle} \frac{(f\mu)(z) }{\mu(z)} = f(x) \qquad \text{for } \mu\text{ almost all }x\in\R. 
	\end{align*}
	
\item If $\dd \mu = w\dd x + \dd \mu\ti s$ is the Lebesgue decomposition of the measure $\mu$, then 
\begin{align}
\lim_{z\to x\sphericalangle} \mu(z) & = w(x) \qquad\text{for Lebesgue almost all }x\in\R, 
\notag\\
\lim_{z\to x\sphericalangle} \mu(z) &= +\infty \qquad \text{for } \mu\ti s\text{ almost all }x\in\R.
\label{BLABLA}
\end{align}
\end{enumerate}
\end{theo}

Part (i) is well-known; it is essentially a version of the Lebesgue differentiation theorem. For a self-contained presentation, see e.g.~\cite[Lemma 1.2]{NONTAN}. A proof of the first statement of part (ii) can be found in \cite[Theorem 11.124]{Rudin}. Although the statement of (3.4) can be found in several places in the literature, we could not find a self-contained proof. We provide a simple proof in the Appendix Section \ref{s-Appendix} below.

%
%

Let $\dd \bM= W\dd\mu$. 
The Lebesgue decomposition $\dd\mu = \dd \mu\ti{ac} + \dd\mu\ti s = w\dd x +\dd\mu\ti s$, $w=\dd\mu/\dd x$ into absolutely continuous and singular parts  yields the corresponding decomposition of the matrix-valued measure $\bM$, 
\[
\dd\bM(x) = W\dd\mu\ti{ac} + W\dd\mu\ti{s}  =\dd\bM\ti{ac}(x) + \dd\bM\ti{s}(x).
\qquad
\]
Defining $W\ti{ac}:=w W = \dd\bM/\dd x$, we can write $\dd\bM\ti{ac}= W\ti{ac}\dd x$.

\begin{theo}\label{t-ID}
Let $\bM$ be a matrix-valued measure and let  $W\ti{ac}$ be its density $\dd\bM/\dd x$ as defined above.

Then $W\ti{ac}$  is determined by the non-tangential limits 
of the Cauchy transform,
\[
W\ti{ac}(x)
= \frac{1}{\pi} \lim_{z\to x\sphericalangle}\im F(z) \quad\text{for Lebesgue a.a.~}x\in \R.
\]
\end{theo}

\begin{rem*}
We encourage the reader to find results about the relation between the boundary values of the Cauchy transform and its matrix-value spectral measure in \cite[Theorems 5.5 and 6.1]{Gesztesy2000}.
\end{rem*}

\begin{proof}
Theorem \ref{t-ID} follows immediately from Theorem \ref{t:BoundValues}, because $\pi^{-1}\im F(z)$ is exactly the Poisson extension of $\bM$ at the point $z$. Then, applying Theorem \ref{t:BoundValues} to entries of $\bM$ we get the result.
\end{proof}

\section{Spectral averaging and mutually singular measures}\label{s-SAveraging}
The spectral averaging formula by Aleksandrov \cite{Aleksandrov} is one of the most curious results in rank one perturbation theory: it states that the average of the spectral measures of the family of the rank one perturbation is the Lebesgue measure on the real line. 

More precisely, if 
\begin{align*}
A_\gamma := A + \gamma \bb \bb^*, \qquad \gamma\in \R
\end{align*}
is a one parameter family of the rank one perturbations (here $\bb:\C\to\cH$ is a rank one operator), and $\mu^\gamma$ are  the corresponding spectral measures (associated with the vector (operator) $\bb$), then  for any Borel measurable 
 function  $f\in L^1(\R)$
\begin{align}\label{e-AInt}
\iint f(x)\dd \mu^\gamma(x)d\gamma = \int f(x)\dd x.
\end{align}
The above identity means that $f\in L^1(\mu^\gamma)$ for almost all $\gamma\in\R$, 
and that the function $\gamma \mapsto \int_\R f(x)\dd \mu^\gamma(x)$ belongs to $L^1(\R)$.

As the averaging formula can be used to obtain spectral and cyclicity information of perturbed operators, we set out to find a generalization of the formula to the finite rank setting.

We first prove a result about averaging  over the line, see Theorem \ref{t-Averaging} below. As one can see from this theorem, integrating over all perturbation parameters $\Gamma$ would give a divergent integral, so one needs to introduce weights to get the convergence.  

In our case it is easy to get the result for the ``cylindrical'' weights, i.e.~$L^1$ functions on the space of $d\times d$ Hermitian matrices that are constant in the direction given by an arbitrary (fixed) positive definite matrix $\Gamma$, see Theorem \ref{t-AvHaar}. As a spectral corollary of this theorem we will get the  Aronszajn--Donoghue type result about mutual singularity of the singular parts for almost all perturbations, see Corollary \ref{c-MutuallySingular} below 

\subsection{Averaging over the line for finite rank perturbations}
Let $A$ and $A\ciG$ be finite rank perturbations given by \eqref{e-AGamma}. Recall that $\bM^{\Gamma}$ is the matrix-valued spectral measure of $A\ciG$ as defined in Section \ref{s-Borel}.

\begin{theo}[Aleksandrov Spectral-type Averaging]\label{t-Averaging}
Let $\Gamma_0$ be a self-adjoint and $\Gamma$ be a positive definite $d\times d$ matrix. Consider a scalar-valued Borel function $f\in L^1(\R)$.
We have
\begin{align}
\label{Disintegration 01}
\int_\R\left(\int_\R f(x) \dd\bM^{\Gamma_0+t\Gamma}(x) \right)\dd t
=
\Gamma^{-1} \int_\R f(x) \dd x .
\end{align}

\end{theo}

\begin{rem*}
Note that for a generalization of \eqref{e-AInt}, the outside integral should be replaced by integration with respect to the Haar measure over the space of complex Hermitian matrices. However, such a left hand side will in general be infinite.
\end{rem*}

Parts of the following proof are an adaptation and generalization of the proofs in \cite[s.~9.4]{SIMREV}.

\begin{proof}
Let us first prove the theorem for the Poisson kernels
\begin{align*}
p_z(x) := \frac{1}{2\pi i}\left(\frac{1}{x-z}-\frac{1}{x-\bar z}\right)\,,
\qquad
z\in \C_+,
\end{align*}
(here $x$ is not the real part of $z$); the rest will be done by the approximation.

For $f=p_z$, $z\in\C_+$ the right hand side of \eqref{Disintegration 01} evaluates to
\[
\Gamma^{-1}\int_\R p_z(x) \dd x
=
2\pi i \Gamma^{-1}
\qquad
\text{for all } z\in \C_+;
\]
this follows because $p_z$ is the Poisson kernel. It can also be done via a standard integration using residues. 

For the evaluation of the left hand side recall the definition of the matrix-valued Cauchy transforms $F$ and $F\ci{\Gamma_0+t\Gamma}$ given in Subsection \ref{ss-MatrixMeasures}. In combination with a variant of Lemma \ref{l-AK}, we obtain
\begin{align}
\label{e-Residue}
&\int_\R p_z(x) \dd\bM^{\Gamma_0+t\Gamma}(x) 
=
(2\pi i)^{-1} \left(F\ci{\Gamma_0+t\Gamma}(z)-F\ci{\Gamma_0+t\Gamma}(\bar z)\right)
\\ \notag 
&
=
(2\pi i)^{-1} \left([F^{-1}( z) +\Gamma_0+   t\Gamma]^{-1}
-
[F^{-1}(\bar z) +\Gamma_0+  t\Gamma]^{-1}\right) =: h_z(t).
\end{align}
Since $\Gamma$ is positive, its positive square root $\Gamma^{1/2}$ is well  defined and one can easily verify that with $\wt F:= \Gamma^{1/2}F\Gamma^{1/2}$ we have
\begin{align*}
F^{-1} +  \Gamma_0+ t\Gamma
&=
\Gamma^{1/2}
( t\OID + \wt F^{-1} +\Gamma^{-1/2}\Gamma_0\Gamma^{-1/2})
\Gamma^{1/2}
\\
&= \Gamma^{1/2}
(t\OID - G)
\Gamma^{1/2}, 
\end{align*}
where $G:= - (\wt F^{-1} +\Gamma^{-1/2}\Gamma_0\Gamma^{-1/2})$. 

Again, we will perform the standard residue calculation with the semi-circle in the upper half-plane. To that end, recall that $F$ is Herglotz, i.e.~$\im F(z)\ge \bO$ for $z\in\C_+$. And since $\Gamma^{1/2}$ is positive, $\wt F$ is Herglotz, too. 

Since for a matrix $T$
\begin{align}
\label{Im inverse}
\im (T^{-1}) = -(T^{-1})^*(\im T) T^{-1}, 
\end{align}
we conclude that the function $-\wt F^{-1}$ is also Herglotz.  The operator $\Gamma^{-1/2}\Gamma_0\Gamma^{-1/2}$ is self-adjoint, therefore the function $G$ is also Herglotz, so $\im G(z)\ge\bO$ for all $z\in\C_+$. 

Since trivially, $F(\bar z )=F(z)^*$, we have  that $G(\bar z)= G(z)^*$, 
so  $\im (G(\bar z)^{-1}) \ge \bO$ for all $z\in\C_+$. 

So when $z\in \C_+$ , then we have for the spectra  $\sigma(G(z))\subset \C_{+}$ and $\sigma(G(\bar z))\subset \C_-$.

We need to evaluate the integral 
\begin{align}
\label{Int h_z}
\int_\R h_z(t)\dd t = \frac{1}{2\pi i}\int_\R \left( (t\bI - G(z) )^{-1} - (t\bI - G(\bar z) )^{-1}  \right) \dd t.
\end{align}
The evaluation is pretty standard residue calculation. We consider the closed contour $\gamma\ci R$ consisting of the interval $[-R,R]$ and the semicircle $S\ci R =\{ w\in\C_+: |w|=R \}$; $R$ is assumed to be sufficiently large, so that $\sigma(G(z))$ is inside the domain bounded by the contour  $\gamma\ci R$. 

Since $\|h_z(w)\|\ci{C^d} = O(R^{-2})$ for $w\in S\ci R$  as $R\to \infty$, we see that 
\[
\int_{S_{_R}}h_z (w)\dd w \to \bO \qquad \text{as } R\to\infty, 
\]
so for sufficiently large $R$ we have
\begin{align}
\label{Int h_z 01}
\int_\R h_z (t)\dd t = \int_{\gamma_{_R}} h_z(t) \dd t. 
\end{align}
Recall (see \eqref{Int h_z}) that $h_z(t) = (2\pi i)^{-1} \left( (t\bI - G(z) )^{-1} - (t\bI - G(\bar z) )^{-1}  \right)$. The second term $(t\bI - G(\bar z) )^{-1}$ is analytic for $t\in\C_+$, so its contribution to the integral 
\eqref{Int h_z 01} is $\bO$. 

Therefore (for sufficiently large $R$)
\begin{align*}
\int_\R h_z (t)\dd t = \int_{\gamma_{_R}} h_z(t) \dd t = \frac{1}{2\pi i} \int_{\gamma_{_R}}  (t\bI - G(z) )^{-1} \dd t = \bI;
\end{align*}
the last equality follows from the Riesz functional calculus. 
This  proves Theorem \ref{t-Averaging} for the Poisson kernels $p_z$.
\medskip

Let us now extend identity \eqref{Disintegration 01} to  wider classes of functions. 

We will need the following simple lemma. 
Let $H(d)$ be the set of  $d\times d$ Hermitian matrices. 

\begin{lem}
\label{l:M^Gamma Poisson bounded}
The matrix measures 
$\bM^\Gamma$ are uniformly Poisson bounded, i.e.~there exists $P<\infty$ (independent of $\Gamma$) so that 
\[
\left\|\int_\R \frac{\dd\bM^\Gamma(x)}{1+x^2}\right\|<P\qquad \forall \Gamma\in H(d). 
\] 
Moreover, if $\Gamma(t)=\Gamma_0 + t\Gamma$ with invertible $\Gamma$ then 
\begin{align*}
\left\|\int_\R \frac{\dd\bM^{\Gamma(t)}(x)}{1+x^2}\right\| = O (t^2) \qquad \text{as }
|t|\to\infty;
\end{align*}
of course, the constants depend on $\Gamma_0$, $\Gamma$. 
\end{lem}

\begin{proof}
Consider function  $p_i(x)= (2\pi i)^{-1}\left( (x-i)^{-1} - (x+i)^{-1} \right)  =(\pi i)^{-1}|x-i|^{-2}$. Using the calculation \eqref{e-Residue} with $\Gamma$ instead of $\Gamma_0 +t\Gamma$ we estimate
\begin{align*}
\frac12\left\|\int_\R \frac{d\bM^\Gamma(x)}{1+x^2}\right\|
&
\le
\left\|
(F(i)^{-1} +   \Gamma)^{-1}\right\|
+
\left\|
(F(-i)^{-1} +  \Gamma)^{-1}\right\|
\\
&
\le
\|\im (F(i)^{-1})^{-1}\|
+\|\im (F(-i)^{-1})^{-1}\| = 2 \|\im (F(i)^{-1})^{-1}\|;
\end{align*}
here, in the second inequality we used the fact that if $\im T$ is invertible, then $T$ is invertible and $\|T^{-1}\|\le\|(\im T)^{-1}\|$. 

The invertibility of $\im (F(i)^{-1})$ follows from identity \eqref{Im inverse} applied to $T=F(i)$ and from the invertibility of $F(i)$. 

To prove the second statement we first notice that for sufficiently large $|t|$ the operators $\bI +\Gamma(t)F(i)$ are invertible and 
\begin{align*}
\|(\bI + \Gamma(t)F(i))^{-1}\| =O(|t|^{-1}) \qquad \text{as } |t|\to \infty;
\end{align*}
here the invertibility of $\Gamma$ is used.
By Lemma \ref{l:Im F vs Im F Gamma} 
\begin{align*}
\|\im F\ci{\Gamma(t)}(i)\| \le \|\im F
(i) \| \|\bI + \Gamma(t) F(i)\|^2, 
\end{align*}
and the second statement follows. 
\end{proof}

Let us now prove that \eqref{Disintegration 01} holds for the class $C\ti c(\R)$ of continuous functions with compact support; in fact we will prove it for a wider class $C\ti{Poiss}$ of Poisson bounded continuous functions. 

Namely, let $\widehat\R$ be the one point compactification of $\R$, where we identify the points $+\infty$ and $-\infty$. Define the space $C\ti{Poiss}=(1+x^2)^{-1}C(\widehat \R)$ equipped with the norm 
\begin{align*}
\|f\|\ti{Poiss}:=\sup_{x\in\R}\{(1+x^2)|f(x)|\}.
\end{align*}

\begin{lem}
\label{l:Cont-disint}
Let $f\in C\ti{Poiss}$. Then the function 
\begin{align*}
\Gamma\mapsto \int_\Gamma f(x) \dd\bM^\Gamma(x) 
\end{align*}
is a continuous function on $H(d)$, and \eqref{Disintegration 01} 
holds for $f\in C\ti{Poiss}$ and all  $\Gamma>\bO$. 
\end{lem}
\begin{proof}
It  easily follows from the  Stone--Weierstra{\ss} theorem, that the linear combinations of $1$ and the Poisson kernels $f_{z_k}$ are dense in $C(\widehat\R)$, 
so the linear combinations of the Poisson kernels $f_{z_k}$ are dense in $C\ti{Poiss}$. 

Let $f\in C\ti{Poiss}$. Take linear combinations $f_n$ of Poisson kernels, such that 
\begin{align}
\label{f_n to f Poisson}
\|f-f_n\|\ti{Poiss} \to 0 \qquad\text{as } n\to \infty.  
\end{align}
The uniform Poisson boundedness of the measures $\bM^\Gamma$ (Lemma \ref{l:M^Gamma Poisson bounded}) implies that 
\begin{align*}
\int_\R f_n(x) \dd\bM^\Gamma(x) \rightrightarrows \int_\R f(x) \dd\bM^\Gamma(x)
\end{align*}
uniformly in $\Gamma\in H(d)$. 

For the Poisson kernel $p_z$
\begin{align*}
2\pi i\int_\R p_z(x) \dd\bM^\Gamma(x) 
& =
F\ci{\Gamma}(z)-F\ci{\Gamma}(\bar z)
\\ \notag
 &=
F( z)(\bI +   \Gamma F( z))^{-1}
-
F(\bar z) (\bI +  \Gamma F(\bar z) )^{-1}, 
\end{align*}
and clearly the right hand side here  continuously depends on $\Gamma$. Therefore the functions $\Gamma \mapsto \int_\R f_n\dd\bM^\Gamma$ are continuous,  
and so is the function $\Gamma \mapsto \int_\R f \dd\bM^\Gamma$, as a uniform limit of continuous functions. 

We already proved that \eqref{Disintegration 01} holds for the Poisson kernels $p_z$, so it holds for the functions $f_n$. The convergence \eqref{f_n to f Poisson} implies that $\|f_n\|\ti{Poiss}\le C<\infty$ uniformly, so 
\begin{align}
\label{bd f_n}
| f_n (x) | \le C (1+x^2)^{-1}\qquad \forall n\ \forall x\in \R. 
\end{align}
Therefore by Lemma \ref{l:M^Gamma Poisson bounded}
\begin{align}
\label{bd Int f_n}
\left\| \int_\R f_n(x) \dd\bM^{\Gamma_0+t\Gamma}(x) \right\| \le C (1+t^2)^{-1}
\end{align}
(with different $C$). 
Then applying the Dominated Convergence Theorem twice we get that
\begin{align*}
\int_ \R \left(\int_\R f(x) \dd\bM^{\Gamma_0+t\Gamma}(x) \right)\dd t 
&=
\lim_{n\to\infty} \int_ \R \left(\int_\R f_n(x) \dd\bM^{\Gamma_0+t\Gamma}(x) \right)\dd t
\\ \notag
& =
\lim_{n\to\infty} \Gamma^{-1} \int_ \R f_n(x) \dd x 
\\ \notag
& =
\Gamma^{-1} \int_ \R f(x) \dd x ; 
\end{align*}
here in the first equality we use the estimate \eqref{bd Int f_n} and the Dominated Convergence Theorem. The second equality is just \eqref{Disintegration 01} for the functions $f_n$, and the last equality follows by the Dominated Convergence Theorem from the estimate \eqref{bd f_n}.

The lemma is proved. 
\end{proof}

To extend \eqref{Disintegration 01} to integrable Borel functions we use the standard reasoning, cf.~\cite[s.~9.4]{cimaross} based on the  Monotone Class Theorem. Recall that a collection $\cT$ of subsets is called a $\pi$-system, if it is closed under finite intersections. We denote by $\sigma(\cT)$ the sigma-algebra generated by $\cT$.

We need the following well-known theorem, see \cite[s.~3.14]{Williams_Prob-mart}. 
 \begin{theo}
\label{t:monotone class}
Let $\cS$ be a set of bounded functions $f:X\to\R$, and $\cT$ be a $\pi$-system such that
\begin{enumerate}
	\item $\cS$ is a real vector space;
	\item the constant function $\1$ belongs to $\cS$;
	\item if $(f_n)\ci{n\ge 1} $ is an increasing sequence of nonnegative functions in $\cS$ such that its limit $f$ 
	\[
	f(x)=\lim_{n\to\infty} f_n(x)
	\]
	is bounded, then $f\in\cS$;
	
	\item  $\cS$ contains all indicator functions $\1\ci I$, $I\in\cT$.
	
	Then $\cS$ contains all bounded $\sigma(\cT)$-measurable functions. 
\end{enumerate}
 \end{theo}

We apply this theorem to the collection $\cT$ of all bounded open intervals $(a,b)$; note that the corresponding sigma-algebra is the Borel sigma-algebra. For the class $\cS$ of functions we take all bounded measurable real functions $g$ on $\R$ such that

\begin{enumerate}
	\item 
		the function  
	\[
	\Gamma\mapsto \int_\R \frac{g(x)}{1+x^2}\dd\bM^\Gamma(x)
	\]
	is Borel measurable; 
	
	\item for all $\Gamma_0 \in H(d)$ and for all positive definite $\Gamma\in H(d)$ the identity \eqref{Disintegration 01} (with integrals being finite) holds for $f$, $f(x)=g(x)/(1+x^2)$.
\end{enumerate}

 Lemma \ref{l:Cont-disint} implies that $C(\widehat \R) \subset \cS$. Assumptions \cond1, \cond2 of Theorem \ref{t:monotone class} are trivially satisfied. The assumption \cond3 is also satisfied: equality of the integrals follows from the Monotone Convergence Theorem (the boundedness of limit implies that the integral  is finite), and the measurability is preserved under  limits (which exist because of monotonicity).

Finally, for any open interval $I$, the function ${\bf 1}\ci I$ can be represented as an increasing limit of non-negative functions $f_n\in C\ti c\subset C(\widehat \R)$. So the assumption \cond4 follows from the fact $C(\widehat \R) \subset \cS$ and from the assumption \cond3 (which as we know is satisfied). 

Thus, the class $\cS$ contains all bounded Borel measurable functions. Taking increasing limits we can see that the class $\cS$ contains all non-negative Borel measurable functions $g$ satisfying  $\int_\R (1+x^2)^{-1}g(x)\dd x<\infty$. Therefore $\cS\supset L^1((1+x^2)^{-1}\dd x)$, and thus Theorem \ref{t-Averaging} is proved in full generality. 
\end{proof}
 

Theorem \ref{t-Averaging} has an  immediate perturbation theoretic consequence:

\begin{cor}\label{c-MutuallySingularA}
Assume the setting of Theorem \ref{t-Averaging}. Let $B$ be a Borel set of zero Lebesgue measure. Then $\bM^{\Gamma_0+t\Gamma}(B) = \OZ$ for Lebesgue a.a.~$t\in \R$.
\end{cor}

\subsection{Averaging over all \texorpdfstring{$\Gamma$}{Gamma}}
Recall that $H(d)$ denotes the complex Hermitian $d\times d$ matrices. Clearly $H(d)$ is a real vector space of  dimension $d^2$; the Frobenius inner product 
\begin{align*}
(S,T)\ci\cF : = \re (\tr(T^*S))=\re(\tr(S^*T))
\end{align*}
makes it into an inner product space. Thus $H(d)$ is isometrically isomorphic to $\R^{d^2}$, so on any subspace of $H(d)$   we can define the standard Lebesgue   measure of  appropriate dimension  (which equals to the appropriately normalized Hausdorff measure).  We use the notation
\[
\Gamma^\perp
:=
\{
S\in H(d):
(S,\Gamma)\ci\cF = 0
\}.
\]
Since $\Gamma^\perp$ has infinite measure, integrating \eqref{Disintegration 01} with $f\ge 0$, $\int_\R f(x)\dd x>0$ over $\Gamma_0\in \Gamma^\perp$ gives us a divergent integral, so Aleksandrov's disintegration formula does not directly generalize to the case of rank $d$ perturbations with $d>1$. To get a generalization we can introduce a weight in the direction of $\Gamma^\perp$.

\begin{theo}
	\label{t-AvHaar}
Let $\Gamma\in H(d)$ be a positive definite matrix. Let $\Phi:\Gamma^\perp \to \R$ be  integrable  (with respect to the Lebesgue measure on $\Gamma^\perp$) and abbreviate $\int_{ \Gamma^\perp} \Phi(\Gamma_0)\dd \Gamma_0 =a$. Then for all $f\in L^1(\R)$ we have
\begin{align*}
\int_{ \Gamma^\perp}\int_\R\int_\R f(x)\Phi(\Gamma_0) [\dd \bM^{\Gamma_0+t\Gamma}(x)] \dd t \dd \Gamma_0
=
a \Gamma^{-1} \int_\R f(x) \dd x, 
\end{align*}
where $\dd\Gamma_0$ denotes the Lebesgue measure of dimension $d^2-1$ on $\Gamma^\perp$. 
\end{theo}

\begin{proof}
	The result follows immediately from Theorem \ref{t-Averaging}  by the Fubini--Tonelli theorems; the measurability of the function $\Gamma\mapsto \int_\R f(x)\Phi(\Gamma_0) [\dd \bM^{\Gamma}(x)]$ was just proved above. 
\end{proof}

Taking a non-vanishing integrable $\Phi\ge 0$ in the above Theorem \ref{t-AvHaar}, we conclude that for any Borel set $B$ of  zero Lebesgue measure $\bM^\Gamma(B)=\bO$ for almost all $\Gamma\in H(d)$. Taking the trace we see that for the scalar measures $\mu^\Gamma:=\tr \bM^\Gamma$ we also have $\mu^\Gamma(B)=0$ for almost all $\Gamma$. This immediately gives us the following Aronszajn--Donoghue type result.

\begin{cor}\label{c-MutuallySingular}
For a singular measure $\nu$ on $\R$ the singular parts $(\mu^\Gamma)\ti s$ of the scalar spectral measures $\mu^\Gamma$ of the operators $A\ci\Gamma$ are mutually singular with $\nu$ for almost all $\Gamma$. In particular, for any fixed $\Gamma_0\in H(d)$ the singular parts of $\mu^\Gamma$ and $\mu^{\Gamma_0}$ are mutually singular for almost all $\Gamma\in H(d)$. 
\end{cor}

\section{Representation theorem}\label{s-Repr}

In this section we assume that the unperturbed  operator $A$ is given in its spectral representation in the weighted space $L^2(\bM)$, where $\bM$ is its matrix-valued spectral measure defined by \eqref{dM 01} and \eqref{d-M}. 

In this representation the operator $\bB$ is given by $(\bB \bc)(t)\equiv\bc$, $\bc\in \C^d$, $t\in \R$; in other words, the operator $\bB$ maps a vector $\bc\in\C^d$ to the function in $L^2(\bM)$ identically equal $\bc$. The adjoint operator $\bB^*$ is then given by 
\begin{align*}
\bB^* f = \int_\R [\dd\bM(t)]f(t) . 
\end{align*} 
As we discussed above in Section \ref{ss-MGamma}, the perturbed operator $A\ci\Gamma = A + \bB\Gamma \bB^*$ is unitarily equivalent to the multiplication $M_s$ by the independent variable $s$ in the weighted space $L^2(\bM^\Gamma)$, where the matrix-valued measure $\bM^\Gamma$ is defined by  \eqref{M^Gamma}. 

We want to find a formula for the spectral representation of $A\ci\Gamma$, i.e.~for a unitary operator $V\ci\Gamma: L^2(\bM)\to L^2(\bM^\Gamma)$ intertwining $A\ci\Gamma$ and $M_s$, 
\[
V\ci\Gamma A\ci\Gamma = M_s V\ci\Gamma. 
\]
\begin{theo}
\label{t:repr-01}
The spectral representation $V\ciG$ takes the form
\begin{align}
\label{Repr-01}
(V\ci\Gamma h
\be)(s) 
=
h(s) \be
-
\Gamma\int_\R \frac{h(t)-h(s)}{t-s} [\dd\bM(t)] \be
\end{align}
for $\be\in \C^d$ and compactly supported $h\in C^1(\R)$.
\end{theo}

\begin{proof}
By the formula \eqref{UnitEquiv 01} with $\bM^\Gamma$ instead of $\bM$ we get that 
\begin{align}
\label{VGamma-02}
V\ci\Gamma \left( (A\ci\Gamma -z\OID)^{-1}\bB\be \right)(s) = (s-z)^{-1}\be, \qquad \be\in\C^d.  
\end{align}
From the resolvent formula 
\begin{align*}
(A-z\OID)^{-1} -(A\ci\Gamma - z\OID)^{-1} = (A\ci\Gamma - z\OID)^{-1} \bB \Gamma\bB^* (A-z\OID)^{-1} 
\end{align*}
we get that for $\be\in \C^d$
\begin{align}
\label{resolvent-03}
(A - z\OID)^{-1} \bB\be & = (A\ci\Gamma-z\OID)^{-1}\bB\be + (A\ci\Gamma - z\OID)^{-1} \bB \Gamma\bB^* (A-z\OID)^{-1}\bB\be \\ \notag
&= (A\ci\Gamma-z\OID)^{-1}\bB\be + (A\ci\Gamma - z\OID)^{-1} \bB \Gamma \be_z, 
\end{align}
where $\be_z\in \C^d$ is given by 
\begin{align*}
\be_z := \bB^* (A-z\OID)^{-1}\bB\be =\int_\R \frac{1}{t-z} [\dd\bM(t)]\be . 
\end{align*}

Therefore, applying \eqref{VGamma-02} to the right hand side of \eqref{resolvent-03} we obtain that 
\begin{align*}
\left(V\ci\Gamma (A - z\OID)^{-1} \bB\be \right) (s) = (s-z)^{-1}\be + (s-z)^{-1}\Gamma\be_z .
\end{align*}
Denoting by $k_z(s):= (s-z)^{-1}$ and noticing that the vector $ (A - z\OID)^{-1} \bB\be$ is represented in $L^2(\bM)$ by the function $k_z \be$, we can rewrite the above identity as 
\begin{align*}
\left(V\ci\Gamma k_z \be\right)(s) = k_z(s)\be + k_z(s) \Gamma \be_z . 
\end{align*}
Since
\begin{align*}
\frac{k_z(t) - k_z(s)}{t-s} = \frac{-1}{(s-z)(t-z)} = - k_z(s)k_z(t), 
\end{align*}
we see that 
\begin{align*}
k_z(s)\be_z = - \int_\R \frac{k_z(t)-k_z(s)}{t-s} [\dd \bM(t)]\be,  
\end{align*}
so \eqref{Repr-01} holds for $h=k_z$. 

Standard approximation reasoning, like the one performed in \cite{LT09} can be applied to complete the proof of the theorem. 
\end{proof}

\section[Vector mutual singularity]{Vector mutual singularity 
	and Aronszajn--Donoghue theorem}\label{s-MMS}
In the rank one setting, Aronszajn--Donoghue theorem asserts the mutual singularity of the singular parts $\mu^\alpha\ti{s}$ and $\mu^\beta\ti{s}$ whenever $\alpha\neq \beta$ (see e.g.~\cite[Theorem 12.2]{SIMREV}, or \cite{Aronszajn, Donoghue} for the original result). In the higher rank setting, this certainly is not true for the canonical scalar-valued spectral measures. In fact, when dealing with the perturbation theory of the singular parts, the proofs from Aronszajn--Donoghue theory encounter serious road blocks. 

Nonetheless, we can obtain a matrix mutual singularity under the assumption that we are perturbing by a positive definite finite rank operator, see Theorem \ref{t-MMS}. Key is an adaption of methods like those in the proof of the necessity of the two weight $(A_2)$-condition for the boundedness of the two-weight Hilbert transform.

\subsection{Vector mutually  singular matrix-valued measures}

\begin{defn}
	\label{d:vector sing}
We say that  matrix-valued measures $\bM$ and $\bN$ are \emph{vector mutually singular} (and write $\bM\perp\bN$) if there exists a  measurable function $\Pi$ whose values are orthogonal projections on $\C^d$ such that
\begin{align*}
\Pi \bM \Pi =\bO , \qquad (\bI - \Pi) \bN(\bI -\Pi) =\bO;  
\end{align*}
here for a measure $\dd \bM= W\dd\mu$ and a measurable matrix-valued function $\Phi$, the measure $\Phi^*\bM \Phi$ is defined as 
\begin{align*}
\Phi^*\bM \Phi (E) = \int_E \Phi(x)^* [\dd\bM(x)] \Phi(x) = \int_E \Phi(x)^* W(x) \Phi(x) \dd\mu(x). 
\end{align*}
for any measurable set $E$. 

Sometimes we will omit ``vector'' and just write \emph{mutually singular}. 
\end{defn}

It is easy to show that the measures $\bM =W\mu$, $\bN=V\nu$ ($W$, $V$ are matrix-valued functions, $\mu$, $\nu$ are scalar measures) are vector mutually singular if and only if one can pick densities $W$ and $V$ (that are originally defined only $\mu$-a.e.~and $\nu$-a.e.~respectively) such that
\begin{align*}
\Ran W(x) \perp \Ran V(x) \qquad \mu\text{-a.e. and }\nu\text{-a.e.}
\end{align*}

\begin{theo}\label{t-MMS}
Let $\bM$ and $\bM^\Gamma$ be matrix-valued spectral measures, defined by \eqref{M^Gamma}, of the operators $A$ and $A\ci\Gamma$ respectively. 
Then their singular parts  $\bM\ti s$ and $\bM^\Gamma\ti s$  satisfy the following vector mutual singularity condition:
\[
\bM\ti{s} \perp \Gamma \bM\ti{s}^\Gamma \Gamma
\qquad\text{or equivalently}\qquad
\Gamma \bM\ti{s} \Gamma\perp \bM\ti{s}^\Gamma. 
\]
\end{theo}

\begin{rem*}
This theorem can be seen as a generalization to the finite rank case of the classical (scalar) Aronszajn--Donoghue theorem; the mutual singularity here is the vector mutual singularity of the matrix spectral measures. 
\end{rem*}

Using this theorem one can obtain an improved result about mutual singularity of the scalar spectral measures of the perturbation.    

Namely, consider the family of operators $A\ci{\Gamma (t)} = A + \bB \Gamma (t) \bB^*$, where $\Gamma (t)= \Gamma_0 + t \Gamma$, $t\in\R$. Let $\bM^{\Gamma_0+t\Gamma}$ be the matrix spectral measure of the operator $A\ci{\Gamma (t)}$ and let $\mu^{\Gamma_0+t\Gamma}=\tr \bM^{\Gamma_0+t\Gamma}$ be its scalar spectral measure. Denote by $(\mu^{\Gamma_0+t\Gamma})\ti s$ the singular part of $\mu^{\Gamma_0+t\Gamma}$.

\begin{theo}
	\label{t:countable}
Let $\Gamma (t)=\Gamma_0 + t\Gamma$, where $\Gamma>0$ and let $\mu^{\Gamma_0+t\Gamma}$ be the scalar spectral measures of $A\ci{\Gamma (t)}$. For an   arbitrary singular Radon measure $\nu$ on $\R$, 
\[
\nu \perp \mu^{\Gamma_0+t\Gamma}\ti s
\] 
for all except maybe countably many $t\in\R$. 
\end{theo}

\begin{rem}
\label{r:mutual singularity 02}
Corollary \ref{c-MutuallySingularA} implies that the singular measure $\nu$ is mutually singular with $\mu^{\Gamma_0+t\Gamma}\ti s$ for almost all $t\in\R$. 
The above Theorem \ref{t:countable} strengthens this result. 
\end{rem}

\begin{proof}[Proof of Theorem \ref{t:countable} (assuming Theorem \ref{t-MMS})]
Since $A_{t_2} = A_{t_1} + (t_2-t_1) \bB \Gamma \bB^*$, Theorem \ref{t-MMS} implies that we can pick densities $W_{t_k}$, $k=1,2$ of the measures $\bM^{\Gamma(t_k)}$ such that
\begin{align*}
\Ran W_{t_1} (x) \perp  \Gamma\Ran W_{t_2} (x) \qquad \text{for }\mu^{t_1}\ti s + \mu^{t_2}\ti s \text{ almost all }x. 
\end{align*}
We introduce an equivalent inner product $(\fdot, \fdot)\ci\Gamma$ in $\C^d$, $({\bf x}, {\bf y})\ci\Gamma = (\Gamma {\bf x},{\bf y})\ci{\C^d}$. (Since $\Gamma>0$, this inner product defines a norm on $\C^d$ that is equivalent to the standard norm.) So the above orthogonality condition just means that the ranges are orthogonal in the inner product $(\fdot, \fdot)\ci\Gamma$, 
\begin{align}
\label{W-Gamma-orthogonality}
\Ran W_{t_1} (x) \perp\ci\Gamma \Ran W_{t_2} (x) \qquad \text{for }\mu^{t_1}\ti s + \mu^{t_2}\ti s \text{ almost all }x. 
\end{align}

Consider the space $L^2(\Gamma\nu) = L^2(\Gamma\nu;\C^d)$. 
If for some $t\in\R$ the measure $\mu^{\Gamma_0+t\Gamma}\ti s$ is not mutually singular with $\nu$ (i.e.~$\mu^{\Gamma_0+t\Gamma}\ti s$ has a non-trivial part that is absolutely continuous with respect to $\nu$), then there exists non-trivial $f_t\in L^2(\Gamma\nu)$ such that 
\begin{align}
\label{f in Ran W}
f_t(x) \in \Ran W_t(x) \qquad \text{for }\nu \text{ almost all }x.
\end{align}

Let $t_1, t_2\in\R$ be such that $\mu^{t_k}\ti s \not \perp\nu$, $k=1,2$, and let $f_{t_k}\in L^2(\Gamma\nu)$  be a non-trivial functions satisfying \eqref{f in Ran W}. Then \eqref{f in Ran W} together with the orthogonality condition \eqref{W-Gamma-orthogonality} implies that $f_{t_1}$ and $f_{t_2}$ are orthogonal in $L^2(\Gamma\nu)$. The separability of the space $L^2(\Gamma\nu)$ immediately implies the conclusion of the theorem. 
\end{proof}
\subsection{Matrix \texorpdfstring{$A_2$}{A<sub>2} condition}
For a matrix-valued measure $\bM$ and $z\in\C\setminus\R$ denote by $\bM(z)$ its Poisson extension, 
\begin{align*}
\bM(z) = \frac1\pi\int_\R \frac{\im z}{|z-s|^2} \dd\bM(s).
\end{align*}

Consider matrix-valued spectral measures $\bM$ and $\bM\uG$  given by \eqref{d-M} of the operators $A$ and $A\ciG$ respectively. 

We say that a pair of matrix measures $\bM$, $\bN$  satisfies the joint Poisson matrix $A_2$ condition, and write $(\bM, \bN)\in (A_2)$ if 
\begin{align}
\label{A_2 01}
\sup_{z\in\C_+} \|\bM(z)^{1/2} \bN(z)^{1/2}\|^2 =: [\bM,\bN]\ci{A_2} <\infty.
\end{align}
The constant $[\bM,\bN]\ci{A_2}$ is called the joint (Poisson) $A_2$ characteristic of the pair $\bM$, $\bN$.

\begin{rem}
\label{r:A_2 commutes}
Since $(\bM(z)^{1/2} \bN(z)^{1/2})^* =  \bN(z)^{1/2} \bM(z)^{1/2}$  the order of terms $\bM(z)^{1/2}$ and  $\bN(z)^{1/2}$ in  \eqref{A_2 01} is not essential, and $[\bM,\bN]\ci{A_2} = [\bN,\bM]\ci{A_2}$
\end{rem}

\begin{rem}
\label{r:A_2 monotonicity}
The matrix $A_2$ condition is monotone in the measures $\bM$ and $\bN$. Namely, if $\wt\bM\le\bM$ and $\wt\bN\le\bN$, then 
\[
\|\wt\bM(z)^{1/2} \wt\bN(z)^{1/2}\|^2 \le \|\bM(z)^{1/2} \bN(z)^{1/2}\|^2 . 
\]
Therefore, if $(\bM, \bN)\in (A_2)$ then $(\wt\bM, \wt\bN)\in (A_2)$ and $[\wt\bM,\wt\bN]\ci{A_2} \le [\bM,\bN]\ci{A_2}$.
\end{rem}

\begin{theo}
\label{t: M-M_Gamma A_2}
Let  $\bM$ and $\bM^\Gamma$  be the matrix-valued spectral measures  (given by \eqref{d-M}) of the operators $A$ and $A^\Gamma$ respectively. 
Then the  measures $\bM$ and $\Gamma \bM^\Gamma \Gamma$ satisfy the matrix $A_2$ condition with 
$[\bM, \Gamma \bM^\Gamma\Gamma]\ci{A_2} \le (8/\pi)^2$, 
\begin{align}
\label{M-M_Gamma A_2}
\| \bM(z)^{1/2} (\Gamma \bM^\Gamma(z) \Gamma)^{1/2}\|\le 8/\pi \qquad \forall z\in \C_+. 
\end{align}
\end{theo}

\begin{rem*}
Since for an operator $T$ the identity $\|T\|^2=\|T^*T\|=\|TT^*\|$ holds, we can write 
\begin{align*}
\| \bM(z)^{1/2} (\Gamma \bM^\Gamma(z) \Gamma)^{1/2}\|^2 & = \| \bM(z)^{1/2}   \Gamma \bM^\Gamma(z) \Gamma  \bM(z)^{1/2} \|
\\
&= \| \bM(z)^{1/2}   \Gamma \bM^\Gamma(z)^{1/2}\|^2 .
\end{align*}
So, one can put $\| \bM(z)^{1/2}   \Gamma \bM^\Gamma(z)^{1/2}\|$ on the left hand side of \eqref{M-M_Gamma A_2}. 

The above identity also implies that one can place $\Gamma$  with $\bM$, i.e.~that $[\bM, \Gamma \bM^\Gamma\Gamma]\ci{A_2} =[\Gamma \bM \Gamma , \bM^\Gamma]\ci{A_2}$. 
\end{rem*}

\begin{proof}[Proof of Theorem \ref{t-MMS}]
Let us show how Theorem \ref{t: M-M_Gamma A_2} implies Theorem \ref{t-MMS}. By part (i) of Theorem \ref{t:BoundValues} we have that for a Radon measure $\mu\ge 0$ on $\R$ and $f\in L^1(\mu)$ 
\begin{align}
\label{Boundary values fmu/mu}
\frac{(f\mu)(z) }{\mu(z)} \to f(x) \qquad \text{for } \mu\text{ almost all }x\in \R
\end{align}
as $z\to x$ non-tangentionally; here recall $\mu(z)$ and $(f\mu)(z)$ are the respective Poisson extension of the measures $\mu$ and $f\mu$ to the point $z\in \C\setminus \R$. 

By part (ii) of Theorem \ref{t:BoundValues} we know that for a singular measure $\mu\ti s$ the non-tangential limit 
\begin{align}
\label{Boundary values mu_s}
\lim_{z\to x\sphericalangle}\mu\ti s(z) =+\infty \qquad \mu\ti s\text{-a.e.~}x\in \R.
\end{align}
By the monotonicity of the $A_2$ condition, see Remark \ref{r:A_2 monotonicity}, we conclude that 
\begin{align*}
\| \bM\ti s(z)^{1/2} (\Gamma \bM^\Gamma\ti s(z) \Gamma)^{1/2}\|\le 8/\pi \qquad \forall z\in \C_+.
\end{align*}
We can rewrite 
\begin{align}
\label{M M^Gamma}
\| \bM\ti s(z)^{1/2} (\Gamma \bM^\Gamma\ti s(z) \Gamma)^{1/2}\| =\mu\ti s(z) 
\left\| \left(\frac{\bM\ti s(z)}{\mu\ti s(z)}\right)^{1/2} \left(\Gamma \frac{ \bM^\Gamma\ti s(z) }{\mu\ti s(z)}\Gamma\right)^{1/2} \right\|.
\end{align}
By \eqref{Boundary values fmu/mu}  we have
\begin{align*}
\lim_{z\to x\sphericalangle} \frac{\bM\ti s(z)}{\mu\ti s(z)} = W(x), 
\qquad \lim_{z\to x\sphericalangle} \frac{ \bM^\Gamma\ti s(z) }{\mu\ti s(z)} 
= \frac{\dd\mu^\Gamma\ti s}{\dd\mu\ti s}(x) W^\Gamma (x) \qquad \mu\ti s\text{-a.e.}
\end{align*}

If the measures $\bM\ti s$ and $\Gamma\bM^\Gamma\ti s \Gamma$ are not vector mutually singular, then there exists a Borel set $E\subset \R$, $\mu\ti s (E)>0$ such that 
\begin{align*}
\frac{\dd\mu^\Gamma\ti s}{\dd\mu\ti s}(x) > 0, \qquad \Ran W(x) \not\perp \Ran 
\left( \Gamma W^\Gamma (x) \Gamma \right) \qquad \mu\ti s \text{-a.e.~on }E. 
\end{align*}
Therefore $W(x)^{1/2}(\Gamma W^\Gamma(x)\Gamma)^{1/2}\ne\bO$  $\mu\ti s$-a.e.~on $E$, and 
it follows from \eqref{M M^Gamma} and \eqref{Boundary values mu_s} that
\begin{align*}
\lim_{z\to x\sphericalangle} \| \bM\ti s(z)^{1/2} (\Gamma \bM^\Gamma\ti s(z) \Gamma)^{1/2}\| =\infty\qquad  \text{for }\mu\ti s\text{-a.a.~} x \in E. 
\end{align*}
But this contradicts \eqref{M-M_Gamma A_2}, and thereby proves Theorem \ref{t-MMS} (modulo Theorem \ref{t: M-M_Gamma A_2}). 
\end{proof}

\subsection{Uniform bounds on some integral operators}
\label{ss:Bounds T_e} To prove Theorem \ref{t: M-M_Gamma A_2} we need to prove uniform bounds for some integral operators. 

For an integral operator $Tf(s) =\int_\R K(s,t) f(t) \dd t$ with bounded kernel $K$ and a matrix-valued measure $\bM=W\mu$, define the operator $T^\bM$, acting on vector-valued functions by 
\begin{align*}
T^\bM f(s)= \int_\R K(s,t)[\dd \bM(t)] f(t)= \int_\R K(s,t) W(t)f(t)\dd\mu(t).  
\end{align*}
We assumed that $K$ is bounded, so everything is well defined say for bounded compactly supported functions.

For $\e>0$ denote by $T_\e$ the integral operator with kernel $1/(s-t\pm i\e)$, and let $T^\bM_{\pm \e}$ denote its vector version with matrix measure $\bM$.

\begin{theo}\label{t-TeUni}
Let $\bM$ and $\bM^\Gamma$ be matrix-valued spectral measures, defined by \eqref{M^Gamma}, of the operators $A$ and $A\ci\Gamma$ respectively. 

Then operators $T_{\pm\e}^{\bM}: L^2(\bM)\to L^2( \Gamma \bM^\Gamma \Gamma)$ are (uniformly in $\e$) bounded with norm at most $2$.
\end{theo}

\begin{proof}
Take a scalar $h\in C_0^1(\R)$, and ${\bf c} \in \C^d$. From the representation formula in Theorem \ref{t:repr-01} we get that for $a\in (0,\infty)$ 
\begin{align*}
  V\ciG h{\bf c} - e^{ias}  V\ciG (e^{-iat}h{\bf c}) =
\Gamma \int_\R \left(1-e^{ia(s-t)} \right) \frac{h(t)}{s-t} [\dd\bM(t)] {\bf c};
\end{align*}
note that the kernel $\left(1-e^{ia(s-t)} \right)/(s-t)$ is bounded, so the integral is well-defined. 

Recall that $  V\ciG$ is a unitary operator from $L^2(\bM)$ to $L^2(\bM^\Gamma)$ and notice that multiplication by $e^{iax}$ is a unitary operator on both $L^2(\bM)$ and $L^2(\bM^\Gamma)$. Together with the previous equality we obtain
\begin{align*}
\left\|
\Gamma \int_\R \left(1-e^{ia(\fdot-t)} \right) \frac{h(t)}{\fdot-t} [\dd\bM(t)] {\bf c}
\right\|_{L^2(\bM^\Gamma)}
\le
2\|h{\bf c}\|\ci{L^2(\bM)}.
\end{align*}

The above inequality holds for all $a\ne0$, so if we average the integrand on the left hand side  in $a$ with any probability measure, we will have the same upper bound. 

Let us average over $a>0$ with the weight $\e e^{-\e a}$; note that $\int_0^\infty \e e^{-\e a}\dd a =1$. 
We get for $\e>0$
\[
\e \int_0^\infty \frac{1-e^{ia(s-t)}}{s-t} e^{-\e a} \dd a
=
\frac{1}{s-t}
-
\frac{i\e}{s-t+i\e}
=
\frac{1}{s-t+i\e}\,, 
\]
so, 
\begin{align*}
\left\|
\Gamma \int_\R  \frac{h(t)}{\fdot-t+i\e} [\dd\bM(t)] {\bf c}
\right\|\ci{L^2(\bM^\Gamma)}
\le
2\|h{\bf c}\|\ci{L^2(\bM)}
\end{align*}
holds uniformly in $\e$.

Since functions of the form $h{\bf c}$ (where $h \in C^1_0$ is a scalar function and ${\bf c}\in\C^d$) are dense in $L^2(\bM)$, we get
\[
\left\|
\Gamma 
 \int_\R  \frac{1}{\fdot-t+i\e}[\dd\bM(t)]  f(t)
\right\|\ci{L^2( \bM\uG)}
\le
2\|f\|\ci{L^2(\bM)}
\]
 for all $f \in L^2(\bM)$, uniformly with respect to $\e$. Since $\| \Gamma g\|\ci{L^2(\bM^\Gamma)} = \|  g\|\ci{L^2(\Gamma\bM^\Gamma\Gamma)}$, the above inequality is exactly the conclusion of the theorem for $T_{+\e}^\bM$. 
 
Averaging over $a<0$ with the weight $\e e^{a\e}$ we get the result for $T_{-\e}^\bM$. 
\end{proof}

For ${\alpha}\in\C\setminus\R$ let $P_{\alpha}$ be the integral operator with kernel $\frac{2\im {\alpha}}{(s-{\alpha})(t-\overline {\alpha})}$, and let $P_{\alpha}^\bM$ be the vector-valued matrix version, as defined in the beginning of this subsection.

\begin{prop}
	\label{p:P_a bound}
Under assumptions of Theorem \ref{t-TeUni} the operators $P_{\alpha}^\bM: L^2(\bM)\to L^2(\Gamma \bM^\Gamma \Gamma)$ are uniformly (in ${\alpha}$) bounded with norm at most $4$. 
\end{prop}

\begin{proof}
For ${\alpha}\in\C\setminus\R$ define $\f_{\alpha}(t):= (t-{\alpha})/(t-\overline {\alpha})$. Using the above operator $T_\e$ with kernel $1/(s-t+i\e)$, we formally define an auxiliary operator $S_{{\alpha},\e}$
\begin{align*}
S_{{\alpha},\e} f = T_\e - M_{\overline{\f_{\alpha}}} T_\e M_{\f_{\alpha}}, 
\end{align*} 
where $M_\f$ is the multiplication operator, $M_\f f = \f f$. 

Let $S_{{\alpha}, \e}^\bM$ be the vector-valued matrix version,  as defined in the beginning of this subsection. Since $|\f_{\alpha}(t)|=1$ on $\R$, the operator $M_{\f_{\alpha}}$ is a unitary operator in both $L^2(\bM)$ and $L^2(\Gamma \bM^\Gamma \Gamma)$, so the operators $S_{{\alpha},\e}^\bM: L^2(\bM)\to L^2(\Gamma \bM^\Gamma \Gamma)$ are uniformly in ${\alpha}$ and $\e$ bounded with the norm at most $4$. 

Computing the kernel of $S_{{\alpha},\e}$ we get 
\[
\frac{1}{s-t+i\e}-\frac{(s-\bar {\alpha})(t-{\alpha})}{(s-{\alpha})(s-t+i\e)(t-\bar {\alpha})}
=
\frac{2i\im {\alpha}(s-t)}{(s-t+i\e)(s-{\alpha})(t-\bar {\alpha})}, 
\]
so for compactly supported $f\in L^2(\bM)$
\begin{align*}
S_{{\alpha},\e}^\bM f\,(s)
=
\int\frac{2i\im {\alpha}(s-t)}{(s-t+i\e)(s-{\alpha})(t-\bar {\alpha})}\,[ \dd\bM(t)] f(t).
\end{align*}
The operators $S^\bM_{{\alpha},\e} : L^2(\bM)\to L^2(\Gamma \bM^\Gamma \Gamma)$ are uniformly bounded, so by an $\e/3$ argument  $S_{{\alpha},\e}^\bM \to i P_{\alpha}^\bM$ as $\e\to0$ in the strong operator topology (the convergence on compactly supported $f$ is trivial, due to the uniform on compact subsets convergence of the kernels). 

Thus we get the desired bound on $P_{\alpha}^\bM$. 
\end{proof}
\subsection{Bounds for operators \texorpdfstring{$P_{\alpha}^\bM$}{P<sub>{alpha}<sup>M} imply matrix \texorpdfstring{$A_2$}{A<sub>2}-condition}

We need the following simple lemma.

\begin{lem}
\label{l:bounds to A_2}
Let $T$ be an integral operator with kernel $K$, $K(s,t)=k_1(s)k_2(t)$ (assume for simplicity that $k_{1}$, $k_2$ are bounded), and let $\bM$, $\bN$ be matrix-valued measures. If operator $T^\bM: L^2(\bM)\to L^2(\bN)$ is bounded, then
\begin{align}
\label{A_2 le norm}
\left\| \left(\int |k_1|^2 \dd\bN \right)^{1/2} \left(\int |k_2|^2 \dd\bM \right)^{1/2} \right\| \le
\|T^\bM\|\ci{L^2(\bM)\to L^2(\bN)}\,.
\end{align}
\end{lem}

\begin{rem*}
	In fact one can show that equality holds in \eqref{A_2 le norm}, but for our purpose the inequality suffices. So, we state and prove the lemma as stated.   
\end{rem*}

\begin{proof}[Proof of Theorem \ref{t: M-M_Gamma A_2}]
The above Lemma \ref{l:bounds to A_2} implies Theorem \ref{t: M-M_Gamma A_2}. Indeed, the kernel of the operator $P_{\alpha}$ is represented as $K(s,t)=2(\im {\alpha})/((s-{\alpha})(t-\overline{{\alpha}})) = k_1(s)k_2(t)$
\begin{align*}
k_1(s)= \frac{(2\im {\alpha})^{1/2}}{s-{\alpha}}, \qquad k_2(t)= \frac{(2\im {\alpha})^{1/2}}{t-\overline {\alpha}} .
\end{align*}
Recall that the Poisson kernel of the upper half plane $\C_+$ is given by
\[
\cP_{\alpha} (t) :=\frac{ \im {\alpha}}{\pi|t-\overline {\alpha}|^2}
\qquad \text{for }{\alpha}\in \C_+, t\in \R,
\]
so $|k_1|^2 = |k_2|^2 = (\pi/2)\cP_{\alpha}$. Therefore
\begin{align*}
\int |k_1|^2 \dd\bM^\Gamma = \frac{\pi}{2}\bM^\Gamma({\alpha}), \qquad \int |k_2|^2 \dd\bM = \frac{\pi}{2}\bM({\alpha}).
\end{align*}
Recalling that  $\| P_{\alpha}^\bM\|\ci{L^2(\bM)\to L^2(\Gamma \bM^\Gamma \Gamma) } \le 4$ by Proposition \ref{p:P_a bound}, we immediately get the conclusion of Theorem \ref{t: M-M_Gamma A_2} from \eqref{A_2 le norm}; recall, see Remark \ref{r:A_2 commutes}, that the order of terms in the definition \eqref{A_2 01} of the matrix $A_2$ condition is not essential. 
\end{proof}

\begin{proof}[Proof of Lemma \ref{l:bounds to A_2}]
Take a unit vector $\be\in\C^d$, $\|\be\|=1$. Define a vector-valued function 	$f=f_\be$ as 
\begin{align*}
f(t) = \left(\int |k_2|^2 \dd\bM \right)^{-1/2} \be \cdot \overline{k_2(t)}; 
\end{align*}
note that $\|f\|\ci{L^2(\bM)} = 1$. Let us compute $T^\bM f$:
\begin{align*}
T^\bM f(s) = k_1(s) \wt \be, 
\end{align*}
where $\wt\be\in\C^d$ is given by 
\begin{align*}
\wt\be = \int k_2(t)[\dd\bM(t)] f(t) = \left(\int |k_2|^2 \dd\bM \right)^{1/2} \be. 
\end{align*}
Therefore we obtain
\begin{align*}
\| T^\bM f\|\ci{L^2(\bN)} = \left\| \left( \int |k_1|^2 \dd\bN \right)^{1/2}\wt\be \,\,   \right\| =
\left\| \left( \int |k_1|^2 \dd\bN \right)^{1/2} \left(\int |k_2|^2 \dd\bM \right)^{1/2} \be \,\,  \right\| .
\end{align*}
Since $\| T^\bM f\|\ci{L^2(\bN)} \le \|T^\bM\|\ci{L^2(\bM)\to L^2(\bN)} \| f\|\ci{L^2(\bM)} = \|T^\bM\|\ci{L^2(\bM)\to L^2(\bN)} $, we get the conclusion of the lemma by taking supremum over all $\be\in\C^d$, $\|\be\|=1$. 
\end{proof}

\appendix
\section{Kato--Rosenblum Theorem}\label{s:KR}

The technique of matrix-valued measures can be used to prove many standard results in the perturbation theory. Thus, 
in this Appendix we present a proof of the Kato--Rosenblum theorem. The easy part  (unitary equivalence of a.c.~parts), see Subsection \ref{s-KR} below,  is a simple corollary of the known facts of the perturbation theory discussed in Sections \ref{s-FINITE}, \ref{s-Borel}. 

The  existence of the wave operators, discussed in Subsection \ref{s:wave} is deduced from our representation theorem (Theorem \ref{t:repr-01}). 
\subsection{Density and an easy part of the Kato--Rosenblum theorem (unitary equivalence of a.c.~parts)}\label{s-KR}
%
The proof below is essentially the proof by Kuroda \cite{Kuroda1963} presented in slightly different language. 

\begin{theo}\label{t-KRFinite}
	Let $A$ and $C$ be self-adjoint operators that differ by a finite rank operator. Then the absolutely continuous parts of $A$ and $C$ are unitarily equivalent.
\end{theo}

\begin{proof}
	Without loss of generality we can assume that $C = A\ciG = A+{\bf B}\Gamma{\bf B}^*$ with invertible $\Gamma$, and that $\Ran\bB$ is cyclic for $A$. 
	
	From Lemma \ref{l:Im F vs Im F Gamma} recall that
	\begin{align}
	\label{ImFGamma}
	\im F\ci\Gamma(z)
	&=
	(\bI+F(z)^*\Gamma)^{-1}
	\im F(z)
	(\bI+\Gamma F(z))^{-1} \\ \notag
	&=
	(\bI+F(z){\Gamma})^{-1}
	\im F(z)
	(\bI+{\Gamma} F(z)^*)^{-1}
	\end{align}
	for $z\in \C\setminus \R$.
	
	By Proposition \ref{p:UnitaryInv 02} it is sufficient to show that
	\begin{align*}
	\dim W\ti{ac} = \dim W^\Gamma\ti{ac},
	\end{align*}
	and in light of \eqref{ImFGamma} it is sufficient to show that the non-tangential boundary values of $\bI+\Gamma F(z)$ (or of $\bI +\Gamma F(z)^*$) are invertible a.e.~on $\R$. So the theorem follows from  Lemma \ref{l:BdValuesF_Gamma} below.  
\end{proof}

\begin{lem}
	\label{l:BdValuesF_Gamma}
	The non-tangential boundary values of $\bI+\Gamma F(z)$ and of $\bI +\Gamma F(z)^*$ (equivalently, of $\bI+ F(z) \Gamma$) as $z\to x\sphericalangle$, $z\in\C_+$, $x\in\R$  are invertible  a.e.~on $\R$ (with respect to  Lebesgue measure). 
\end{lem}

\begin{proof}
	By Lemma \ref{l-AK} the matrices $\bI+\Gamma F(z)$, $\bI + F(z)\Gamma$ are invertible for all $z\in\C_+$, so 
	$\det(\bI+\Gamma F(z))$, $\det(\bI + F(z)\Gamma)$ are  non-trivial (not identically zero) analytic functions in $\C_+$. 
	
	The function $z\mapsto F(z)$, $z\in\C_+$  (i.e.~its matrix entries) has non-tangential boundary values a.e.~on $\T$, so the same holds for $\det(\bI+\Gamma F(z))$. 
	
	By Privalov's theorem, see, for example \cite[Section III.D.3]{Koosis-Hp-1998}, if a non-trivial analytic function $f$ in $\C_+$ has non-tangential boundary values $f(x)$ a.e.~on $\R$, then $f(x)\ne 0$ a.e.%
	\footnote{In \cite[Section III.D.3]{Koosis-Hp-1998} the theorem was stated for the unit disc $\D$, but the standard conformal map gives the result for $\C_+$.  }
	The lemma (and so Theorem \ref{t-KRFinite}) is proved. 
\end{proof}

Combining equation \eqref{ImFGamma} with Theorem \ref{t-ID} and with the above Lemma \ref{l:BdValuesF_Gamma}, we obtain the density of the matrix-valued spectral measure of the perturbed operator.

\begin{lem}\label{l-CorKR}
	With respect to Lebesgue a.e.~$x\in \R$ we have
	\[
	(W\ciG)\ti{ac}(x) =
	\lim_{z\to x\sphericalangle, z\in\C_{\pm}}(\bI +F(z)^* \Gamma )^{-1}
	W\ti{ac}(x)
	\lim_{z\to x\sphericalangle, z\in\C_{\pm}}(\bI +\Gamma F(z))^{-1}
	.
	\]
\end{lem}

\subsection{Kato--Rosenblum theorem: existence of wave operators}
\label{s:wave} 
In this section we prove the hard part of the Kato--Rosenblum theorem, i.e.~the existence of the wave operators 
\begin{align*}
\cW^\Gamma_\pm =\s-lim_{\tau\to\pm\infty} e^{i\tau A_\Gamma}e^{-i\tau A} P\ti{ac}, 
\end{align*}
where $P\ti{ac}$ is the orthogonal projection onto the absolutely continuous spectrum of $A$. 

The proof is rather standard, although we did not see exactly the same proof in the literature. We first establish the existence of the weak wave operators (the limits are in the weak operator topology). Then, using that the operators are unitary, we obtain that the limits also exist in the strong operator topology. The existence of weak limits is deduced from our representation theorem (Theorem \ref{t:repr-01}). 

As we discussed above in Section \ref{ss:Bounds T_e}: if $T_{\pm\e} $ is an integral operator with kernel $(s-t\pm i\e)^{-1}$, then the corresponding operators with $T^\bM_{\pm\e}$ matrix measure $\bM$, $T_{\pm\e}^\bM: L^2(\bM) \to L^2(\Gamma\bM^\Gamma \Gamma)$ are uniformly in $\e$ bounded. 

A natural idea would be to take the limit in weak operator topology (w.o.t.); but for this one needs to show that there is a unique w.o.t.~limit point as $\e\to0$. For the projection on the absolutely continuous part of $L^2(\Gamma \bM^\Gamma\Gamma)$ the result is easy. 

Denote by $(T^\bM_\pm)\ti{ac} f $ the non-tangential boundary values of $\cC(\bM f)$ as $z\to x\in\R$, $z\in\C_\pm$ respectively. By the classical results the non-tangential boundary values of $\cC(\bM f)$ exist a.e.~with respect to the Lebesgue measure. 

\begin{lem}
	\label{l:wotlim-ac}
	In the weak operator topology of $B(L^2(\bM);L^2(\Gamma \bM^\Gamma\ti{ac}\Gamma))$ we have
	\begin{align*}
	\wot-lim_{\e\to 0^+} T^\bM_{\pm\e}=(T^\bM_\pm)\ti{ac}   .
	\end{align*}
\end{lem}
\begin{proof}
	Take $f\in L^2(\bM)$. Since 
	\begin{align*}
	T^\bM_{\pm\e} f \to (T^\bM_\pm)\ti{ac} f \qquad \text{a.e.}, 
	\end{align*}
	as $\e\to 0^+$, \cite[Lemma 3.3]{LT09} says that for any weakly convergent sequence 	$T^\bM_{\e_k} f$, $\e_k\to 0^+$ we have $\w-lim_{k} T^\bM_{\e_k} f = (T^\bM_+)\ti{ac} f $. 
	
	Combining this with the fact that any sequence $T^\bM_{\e_k} f$ has a weakly convergent subsequence, we get the conclusion of the lemma for $T^\bM_{\e}$. The case of $T^\bM_{-\e}$ is treated absolutely the same way. 
\end{proof}

Let $P^\Gamma\ti{ac}$ be the (orthogonal) projection in $L^2(\bM^\Gamma)$ onto its absolutely continuous part $L^2(\bM\ti{ac})$, and let $F_\pm$ be the non-tangential boundary values of $F(z)=[\cC \bM f](z)$ as $z\to x\in \R$, $z\in \C_\pm$ respectively;  recall that the boundary values $F_\pm$ exist Lebesgue a.e.

\begin{lem}
	\label{l:AltRepr-ac}
	For the spectral representation $V\ci{\Gamma}: L^2(\bM)\to L^2(\bM^\Gamma)$ from Theorem \ref{t:repr-01} we have 
	\begin{align}
	\label{AltRepr-ac-01}
	P^\Gamma\ti{ac}V\ci\Gamma f = P^\Gamma\ti{ac}\left((\bI
	+ \Gamma F_\pm) f - \Gamma (T^\bM_\pm)\ti{ac} f\right).
	\end{align}
\end{lem}

\begin{proof}
	Lemma \ref{l:Im F vs Im F Gamma} implies that the multiplication operator $f\mapsto (\bI
	+ \Gamma F_\pm) f$ is a contraction acting $L^2(\bM)\to L^2(\bM^\Gamma\ti{ac})$.

	Denote by $V\ci\Gamma^\e$ the operator, defined on functions of form $h\be$, where $h$ is a scalar function in  $C^1\ti c$, and $\be\in\C^d$, as 
	\begin{align*}
	(V\ci\Gamma^{\pm\e} h
	\be)(s) 
	=
	h(s) \be
	-
	\Gamma\int_\R \frac{h(s)-h(t)}{s-t\pm i\e} [\dd \bM(t)] \be.
	\end{align*}
	It is easy to see that for $h\in C^1\ti c$ the functions $V\ci\Gamma^{\pm\e} h\be$ converge to $V\ci\Gamma h
	\be$ as $\e\to 0^+$ uniformly on compact subsets of $\R$. Therefore for any $f=\sum_{k=1}^n h_k\bc_k$, where $h_k$ are scalar functions in $C^1\ti c(\R)$ and $\bc_k\in\C^d$, and for any bounded compactly supported $g\in L^2(\bM^\Gamma\ti{ac})$ we have
	\begin{align*}
	\lim_{\e\to 0^+} (V\ci\Gamma^{\pm\e}f, g)\ci{L^2(\bM^\Gamma\ti{ac})} = (V\ci\Gamma f, g)\ci{L^2(\bM^\Gamma\ti{ac})} .
	\end{align*}
	The limits in the left hand side give us 
	\begin{align*}
	\left((\bI
	+ \Gamma F_\pm) f - \Gamma (T^\bM_\pm)\ti{ac} f , g \right)\ci{L^2(\bM^\Gamma\ti{ac})}, 
	\end{align*}
	which is exactly the bilinear form of the operator on the right hand side of \eqref{AltRepr-ac-01}. Thus, the bilinear forms of the operators in \eqref{AltRepr-ac-01} coincide on a dense set, so the operators are equal.  
\end{proof}

\begin{lem}
	\label{l:WaveUnitary}
	The multiplication operators $f\mapsto (\bI +\Gamma F_\pm)f$ are unitary operators acting from $L^2(\bM\ti{ac})$ to $L^2(\bM^\Gamma\ti{ac})$. 
\end{lem}

\begin{proof}
	Since $F(\overline z) = F(z)^*$  for $z\in\C\setminus\R$, we can conclude that the non-tangential boundary values of $F(z)$, and $F(z)^*$,  $z\in\C_\pm$ are given by $F_\pm$. Lemma \ref{l:BdValuesF_Gamma} implies that the functions  $\bI +\Gamma F_{\pm}$
	are invertible a.e.~on $\R$. 
	Recall that by Lemma \ref{l-CorKR} we have
	\begin{align*}
	W^\Gamma\ti{ac} = ((\bI + \Gamma F_+)^{*})^{-1} W\ti{ac} (\bI + \Gamma F_+)^{-1} =
	( (\bI + \Gamma F_-)^{*})^{-1} W\ti{ac} (\bI + \Gamma F_-)^{-1}. 
	\end{align*}	
	Since the functions $\bI +\Gamma F_{\pm}$
	are invertible a.e.~on $\R$ (by Lemma \ref{l:BdValuesF_Gamma}), we easily conclude that the corresponding multiplication operators are unitary operators acting from $L^2(\bM\ti{ac})$ to $L^2(\bM^\Gamma\ti{ac})$. 
\end{proof}

For $a\in\R$, let $S_a$ be the multiplication by the function $x\mapsto e^{iax}$, $x\in\R$. 
Denote $\cW^\Gamma(\tau) =  e^{i\tau A_\Gamma}e^{-i\tau A} $. Let $P^{A_\Gamma}\ti{ac}$ be the spectral  projection on the absolutely continuous part of $A\ci\Gamma$ (acting in $L^2(\bM)$), so $V\ci\Gamma P^{A_\Gamma}\ti{ac} = P^\Gamma\ti{ac}V\ci\Gamma$. 
Lemma \ref{l:AltRepr-ac} then  yields 
\begin{align}
\label{W(a)}
V\ci\Gamma P^{A_\Gamma}\ti{ac}\cW^\Gamma(a) f = P^\Gamma\ti{ac}V\ci\Gamma \cW^\Gamma(a) f = (\bI
+ \Gamma F_\pm) f - \Gamma S_a T^\bM_\pm (S_{-a}f).
\end{align}

\begin{lem}
	\label{l:w-wave}
	For any $f\in L^2(\bM\ti{ac})$  we have  convergence in the weak topology of $L^2(\bM^\Gamma\ci{ac})$, 
	\begin{align*}
	\w-lim_{a\to \pm\infty} V\ci\Gamma P^{A_\Gamma}\ti{ac}\cW^\Gamma(a) f = (\bI
	+ \Gamma F_\pm) f. 
	\end{align*}
\end{lem}

\begin{proof}
	Let $ W\ti{ac}$ and $ W^\Gamma\ti{ac}$ be the densities (with respect to the Lebesgue measure) of the absolutely continuous parts of the measures $\bM$ and $\bM^\Gamma$ respectively, 
	\begin{align*}
	W\ti{ac} = \frac{\dd\bM}{\dd x} = W \frac{\dd\mu}{\dd x}, \qquad 
	W^\Gamma\ti{ac} = \frac{\dd\bM^\Gamma}{\dd x} = W^\Gamma \frac{\dd\mu^\Gamma}{\dd x}. 
	\end{align*}

	Take $f\in L^2(\bM\ti{ac})$, $g\in L^2(\bM^\Gamma\ti{ac})$, such that $\wt f:= W\ti{ac} f\in L^2$, $\wt g:=  W^\Gamma\ti{ac} g\in L^2$ (note that such $f$ and $g$ are dense in $L^2(\bM\ti{ac})$ and $L^2(\bM^\Gamma\ti{ac})$ respectively). 
	
	Using \eqref{W(a)} we can write
	\begin{align*}
	\left( V\ci\Gamma P^{A_\Gamma}\ti{ac}\cW^\Gamma(a) f, g  \right)\ci{L^2(\bM^\Gamma)} =
	\left( (\bI + \Gamma F_\pm) f, g  \right)\ci{L^2(\bM^\Gamma\ti{ac})} - \left(\Gamma S_{-a} T_\pm (S_a \wt f), \wt g \right)\ci{L^2};
	\end{align*}
	here by  $T_\pm h$ we denote the non-tangential boundary values of $\cC f (z)$, $z\in\C_\pm$ respectively.
	We should emphasize here that the second inner product is in the non-weighted $L^2$! 
	
	One can easily see that for $h\in L^2$ the functions $T_\pm h$ are just the orthogonal projections of $h$ onto the Hardy spaces $H^2(\C_\pm)$  respectively. Therefore, 
	\begin{align*}
	\lim_{a\to\pm\infty} \| T_\pm (S_a\wt f) \|\ci{L^2} = 0, 
	\end{align*}
	and since $S_a$ are unitary (and so uniformly bounded), 
	\begin{align*}
	\lim_{a\to\pm\infty} \left(\Gamma S_{-a} T_\pm (S_a \wt f), \wt g \right)_{L^2} =0.   
	\end{align*} 
	So, on the dense set of $f$ and $g$ as above, 
	\begin{align*}
	\lim_{a\to\pm\infty} \left(  V\ci\Gamma P^{A_\Gamma}\ti{ac}\cW^\Gamma(a) f , g  \right)_{L^2(\bM^\Gamma)} = 
	\left( (\bI
	+ \Gamma F_\pm) f, g \right)_{L^2(\bM^\Gamma\ti{ac})} .
	\end{align*}
	Together with the uniform boundedness of the operators $V\ci\Gamma P^{A_\Gamma}\ti{ac}\cW^\Gamma(a)$ this implies the desired weak convergence. 
\end{proof}

To prove the strong convergence we need the following simple and well-known lemma. 
\begin{lem}
	\label{l:Weak-Strong}
	Let $x(t)$,  be a family of vectors in a Hilbert space such that $\w-lim_{t\to t_0} x(t)  =x$ and $\lim_{t\to t_0} \|x(t)\|=\|x\|$. Then $x(t)$ converges to $x$ in norm, 
	\begin{align*}
	\lim_{t\to t_0} \|x(t) -x\| =0.  
	\end{align*}
\end{lem}

The proof is very simple, we leave it to the readers as an exercise.

The existence of the wave operators follows from the theorem below. 
\begin{theo}
	\label{t:wave}
	For any $f\in L^2(\bM\ti{ac})$
	\begin{align*}
	\s-lim_{a\to \pm\infty} V\ci\Gamma P^{A_\Gamma}\ti{ac}\cW^\Gamma(a) f = (\bI + \Gamma F_\pm) f. 
	\end{align*}
\end{theo}

\begin{proof}
	By Lemma \ref{l:w-wave} we already have weak convergence. 
	
	For a function with values in a Hilbert space, the weak convergence $\w-lim_{t\to t_0} x(t) =x$ implies that $\liminf_{t\to t_0}\|x(t)\|\ge \|x\|$. But the operators $V\ci\Gamma P^{A_\Gamma}\ti{ac}\cW^\Gamma(a)$ are contractions, and the multiplication by  $(\bI + \Gamma F_\pm)$ is a unitary operator from $L^2(\bM\ti{ac})$ to $L^2(\bM^\Gamma\ti{ac})$. Therefore 
	\begin{align*}
	\limsup_{a\to\pm\infty} \| V\ci\Gamma P^{A_\Gamma}\ti{ac}\cW^\Gamma(a) f \|\ci{L^2(\bM^\Gamma)} \le \|f\|\ci{L^2(\bM\ti{ac})} = \| (\bI + \Gamma F_\pm) f\|\ci{L^2(\bM^\Gamma\ti{ac})}
	\end{align*}
	so we have equality for the limit. 
\end{proof}

\section{Proof of the statement  \eqref{BLABLA} of Theorem \ref{t:BoundValues} }\label{s-Appendix}

The second part of statement   \cond2 of Theorem  \ref{t:BoundValues}, see \eqref{BLABLA},  appears a lot in the literature, but we were not able to find a good reference to a self-contained proof of this fact. Most sources just refer without any specifics to classical monographs, where after some time one can extract the needed facts from a proof of a more general result.  

So, for convenience of the reader we present here a simple self-contained proof of this statement.

Let $\cD_n$ be the collection of dyadic intervals of length $2^{-n}$, 
\[
\cD_n := \left\{ 2^{-n}\left( [0,1) + k \right): k\in \Z \right\}, 
\]
and let $\cD:=\bigcup_{n\in\Z}$ be the collection of all dyadic intervals. 

For a Radon measure $\mu$ on $\R$ define the ``conditional expectation'' $\E_n\mu$ by 
\begin{align*}
\E_n\mu(x) = \sum_{I\in\cD_n} \left( \mu(I)/|I|  \right)\1\ci I(x), 
\end{align*}
and the dyadic lower density $\underline D\ut{d}\mu$ as 
\begin{align*}
\underline D\ut{d}\mu (x):= \liminf_{n\to\infty} \E_n(x).  
\end{align*}

\begin{lem}
\label{l:density-01}
Let for a Borel set $E\subset\R$
\begin{align}
\label{LD-01}
\underline{D}\ut d \mu(x)< \alpha\qquad \forall x\in E,  
\end{align}
where $0<\alpha<\infty$. 

Then $\mu(E)\le \alpha |E|$. 
\end{lem}
\begin{proof}
We only need to consider the case $|E|<\infty$, because otherwise the inequality is trivial.
 
Take $\e>0$. By the regularity of the Lebesgue measure there exists an open set $U\supset E$ such that $|U|\le |E|+\e$.
Let $\cE$ be the collection of maximal (by inclusion) intervals $I\in\cD$, $I\subset U$ such that $\mu(I)<\alpha |I|$. 
Note that the intervals in $\cD$ are disjoint, and the collection $\cE$ is countable. 

By the assumption \eqref{LD-01} we have $E\subset \bigcup_{I\in\cE}I=:\wt E$, and by the construction $\wt E\subset U$. Therefore  
\begin{align*}
\mu(E)\le \mu(\wt E) = \sum_{I\in\cE} \mu(I)<\alpha \sum_{I\in\cE} |I| = \alpha |\wt E| \le \alpha |U| \le \alpha (|E|+\e) ,  
\end{align*} 
and since $\e>0$ is arbitrary, we get the conclusion of the lemma. 
\end{proof}

\begin{cor}
\label{c:LD-01}
Let $X_\alpha :=\left\{ x\in \R : \underline D\ut d \mu(x)<\alpha  \right\}$, $\alpha<\infty$. Then $\mu\ti s (X_\alpha) = 0$.  
\end{cor}
\begin{proof}
If $\mu\ti s (X_\alpha)>0$, then there exists a Borel set $E\subset X_\alpha$, $|E|=0$ such that $\mu(E)>0$. But that contradicts the above Lemma \ref{l:density-01}. 
\end{proof}

\begin{proof}[Proof of the statement in \eqref{BLABLA}]
Let $X_n$ denote $X_\alpha$ from the above Corollary \ref{c:LD-01} with $\alpha=n$.  Since $X:=\{x\in \R: \underline D\ut d \mu(x)<\infty \}= \bigcup_{n\in\N} X_n$, the above Corollary \ref{c:LD-01} implies that $\mu\ti s(X)=0$.   But this means exactly that $\underline D\ut d\mu(x)=\infty$ $\mu\ti s$-a.e.

The trivial inequality 
\begin{align*}
\underline D\ut d \mu(x) \le C \liminf_{z\to x \sphericalangle} \mu (z), 
\end{align*}
where $C$ is an absolute constant, gives us the desired statement. 
\end{proof}

\textbf{Acknowledgement.} The authors thank the referee for the useful suggestions.

\providecommand{\bysame}{\leavevmode\hbox to3em{\hrulefill}\thinspace}
\providecommand{\MR}{\relax\ifhmode\unskip\space\fi MR }
\providecommand{\MRhref}[2]{%
  \href{http://www.ams.org/mathscinet-getitem?mr=#1}{#2}
}
\providecommand{\href}[2]{#2}


\begin{thebibliography}{10}


\bibitem{kurasovbook}
S.~Albeverio, P.~Kurasov, \emph{Singular Perturbations of Differential Operators}, London Math.~Soc. Lecture Note Series {\bf 271}, Cambridge University Press, Cambridge, UK, 2000. 

\bibitem{Aleksandrov}
A.B.~Aleksandrov, \emph{Multiplicity of boundary values of inner functions}, Izv.~Akad.~Nauk Armyan. SSR
Ser.~Math.~{\bf 22} (1987), no.~5, 490--503, 515.



\bibitem{Aronszajn}
N.~Aronszajn, \emph{On a Problem of Weyl in the Theory of Singular Sturm--Liouville Equations}, Amer.~J.~Math.~\textbf{79} (1957), no.~3,
  597--610. 

\bibitem{BirmanSol-book_1987}
M.Sh.~Birman, M.Z.~Solomjak, \emph{Spectral theory of selfadjoint
	operators in {H}ilbert space}, Mathematics and its Applications (Soviet
Series), D.~Reidel Publishing Co., Dordrecht, 1987, Translated from the 1980
Russian original by S.~Khrushch\"ev and V.~Peller. 


%
\bibitem{deBranges}
L.~de Branges, \emph{Perturbations of self-adjoint transformation}, Amer.~J.~Math., 
{\bf 84} (1962),  no.~4, 543--560.


\bibitem{cimaross}
J.A.~Cima, A.L.~Matheson, W.T.~Ross, \emph{The {C}auchy
  transform}, Mathematical Surveys and Monographs, vol.~125, Amer.~Math.~Soc., Providence, RI, 2006. 

\bibitem{Donoghue}
W.F.~Donoghue, \emph{On the Perturbation of Spectra}, Comm.~Pure Appl.~Math., 
{\bf 18} (1965),  559--579.

%

\bibitem{Gesztesy2000}
F.~Gesztesy, E.~Tsekanovskii, \emph{On Matrix-Valued Herglotz Functions}, Math.~Nachr., \textbf{218} (2000), iss.~1, 61--138.






 \bibitem{kato1957}
T.~Kato, \emph{On finite-dimensional perturbations of self-adjoint operators}, J.~Math.~Soc.~Japan \textbf{9} (1957) no.~2, 239--249.

   
 \bibitem{katokuroda}
T.~Kato and S.T.~Kuroda, \emph{The abstract theory of scattering}. Rocky Mountain J.~of Math., \textbf{1} (1971) no.~1, 127--171.
   
\bibitem{Koosis-Hp-1998}
P.~Koosis, \emph{Introduction to {$H_p$} spaces}, second ed., Cambridge
Tracts in Mathematics, vol.~115, Cambridge University Press, Cambridge, 1998,
with two appendices by V.P.~Havin [Viktor Petrovich Khavin]. 

\bibitem{Kuroda1967}
S.T.~Kuroda, \emph{An abstract stationary approach to perturbation of continuous spectra and scattering
theory.} J.~Anal.~Math.~\textbf{20} (1967) 57--117.


\bibitem{Kuroda1963}
S.T.~Kuroda, \emph{Finite-dimensional perturbation and a representation of scattering operator.} Pacific J.~Math.~\textbf{13} (1963) 1305--1318.

%
\bibitem{LT09}
C.~Liaw, S.~Treil, \emph{{Rank one perturbations and singular integral
  operators}}, J.~Funct.~Anal., \textbf{257} (2009) no.~6,
  1947--1975.
  
%
%
%
\bibitem{NONTAN}
A.~Poltoratski{\u\i}, \emph{Boundary behavior of pseudocontinuable
  functions}, Algebra i Analiz \textbf{5} (1993), no.~2, 189--210, Engl.~translation in St.~Petersburg Math.~J., 5(2): 389--406, 1994.

%
%

\bibitem{Rudin}
W.~Rudin, \emph{Real and complex analysis}, third ed., McGraw-Hill Book Co., New York, 1987.


\bibitem{SIMREV}
B.~Simon, \emph{Spectral analysis of rank one perturbations and applications}, in Mathematical Quantum Theory. II.~Schr\"odinger Operators, Vancouver, BC (1993), in CRM Proc.~Lecture Notes, {\bf 8}, Amer.~Math.~Soc., Providence, RI (1995), 109--149. 

\bibitem{weyl}
H.~Weyl, \emph{\"{U}ber gew\"ohnliche {D}ifferentialgleichungen mit
  {S}ingularit\"aten und die zugeh\"origen {E}ntwicklungen willk\"urlicher
  {F}unktionen}, Math.~Ann., \textbf{68} (1910), no.~2, 220--269.


\bibitem{Williams_Prob-mart}
D.~Williams, \emph{Probability with martingales}, Cambridge Mathematical
Textbooks, Cambridge University Press, Cambridge, 1991.

\bibitem{Yafaev1992}
D.R.~Yafaev, \emph{Mathematical Scattering Theory: General Theory}, Translation of Mathematical Monographs, \textbf{105}, Amer.~Math.~Soc., Providence, RI, 1992.

\end{thebibliography}
\end{document}